\documentclass{amsart}
\usepackage{amsmath}
\usepackage{amsthm}
\usepackage{amssymb}
\usepackage{mathrsfs}
\usepackage{hyperref}
\usepackage{chngcntr}
\usepackage{appendix}
\usepackage{pifont}
\usepackage{tikz}
\usetikzlibrary{matrix,arrows}
\usepackage{stackengine}
\stackMath
\setstackgap{S}{1pt}
\usepackage{wasysym}

\counterwithin{equation}{section}

\makeatletter
\newtheorem*{rep@theorem}{\rep@title}
\newcommand{\newreptheorem}[2]{%
\newenvironment{rep#1}[1]{%
 \def\rep@title{#2 \ref{##1}}%
 \begin{rep@theorem}}%
 {\end{rep@theorem}}}
\makeatother

\newtheorem{thm}{Theorem}[section]
\newreptheorem{thm}{Theorem}
\newtheorem{prop}[thm]{Proposition} 
\newreptheorem{prop}{Proposition}

\newtheorem{lem}[thm]{Lemma}
\newreptheorem{lem}{Lemma}
\newtheorem{cor}[thm]{Corollary}

\theoremstyle{definition}
\newtheorem{dfn}[thm]{Definition}
\newtheorem{exmpl}[thm]{Example}
\newtheorem{?}[thm]{Question}

\theoremstyle{remark}
\newtheorem{rmk}[thm]{Remark}

\newcommand{\ds}{\displaystyle}
\newcommand{\tql}{\textquotedblleft}
\newcommand{\tqr}{\textquotedblright}
\newcommand{\noin}{\noindent}
\newcommand{\mc}{\mathcal}
\newcommand{\mb}{\mathbb}
\newcommand{\mbf}{\mathbf}

\newcommand{\peq}{\preceq}

\newcommand{\nci}{\Shortstack{. . . .}} 
\newcommand{\bb}{$\bullet\bullet$}

\newcommand{\car}{\curvearrowright}

\begin{document}
\title[Graph products of multipliers]{On graph products of multipliers and the Haagerup property for $C^*$-dynamical systems}

\author{Scott Atkinson}
\thanks{The author received partial support from NSF Grant \# DMS-1362138.}

\address{Vanderbilt University, Nashville, TN, USA}
\email{scott.a.atkinson@vanderbilt.edu}

\begin{abstract}

We consider the notion of the graph product of actions of groups $\left\{G_v\right\}$ on a $C^*$-algebra $\mc{A}$ and show that under suitable commutativity conditions the graph product action $\bigstar_\Gamma \alpha_v: \bigstar_\Gamma G_v \car \mc{A}$  has the Haagerup property if each action $\alpha_v: G_v \curvearrowright \mc{A}$ possesses the Haagerup property.  This generalizes the known results on graph products of groups with the Haagerup property.  To accomplish this, we introduce the graph product of multipliers associated to the actions and show that the graph product of positive definite multipliers is positive definite.  
%The proof of this fact calls upon an inductive argument and combinatorial tools for graph products previously introduced by the author.  
These results have impacts on left transformation groupoids and give an alternative proof of a known result for coarse embeddability. We also record a cohomological characterization of the Haagerup property for group actions.

%\smallskip
%
%\noindent\textbf{Keywords.} quasidiagonal, graph products, max tensor products

\end{abstract}
\maketitle

\section{Introduction}

The Haagerup property is an important approximation property for groups and has been the subject of intense study since its appearance in Haagerup's article \cite{haagmap}. The Haagerup property was first imported into operator algebras by Choda in \cite{chodahp} for the setting of II$_1$-factors. Dong introduced the Haagerup property for $C^*$-algebras much later in \cite{dong}. More recently, Dong-Ruan introduced the Haagerup property in the context of Hilbert $C^*$-modules in \cite{donrua}. In the same article, Dong-Ruan defined the Haagerup property for the action of a discrete group $G$ on a unital $C^*$-algebra $\mc{A}$. Since the trivial action of a group has the Haagerup property if and only if the group has the Haagerup property, this treatment for group actions generalizes the classical notion of the Haagerup property for groups--see \cite{ccjjv} for a survey on the group setting.  B\'edos-Conti further considered the group action context in \cite{bedcon}. The definition of the Haagerup property for $C^*$-dynamical systems involves the notion of positive definite multipliers for the group action: $\mc{Z}(\mc{A})$-valued maps on $G$ that satisfy a positivity condition involving the group action--see Definition \ref{pdmul}.  Such multipliers were first introduced by Anantharaman-Delaroche in \cite{ad} in consideration of amenable group actions.

Graph products unify the notions of free products and direct/tensor products.   In particular, given a simplicial graph $\Gamma = (V,E)$ assign an object (e.g. group, ring, algebra, etc.) to each vertex. If there is an edge between two vertices then the two corresponding objects commute with each other in the graph product; if there is no edge between two vertices then the two corresponding objects have no relations with each other within the graph product.  Such products initially appeared in the group theory context, and the most well-known examples are right-angled Artin groups (graph products of $\mb{Z}$) and right-angled Coxeter groups (graph products of $\mb{Z}/2\mb{Z}$).  See the following (woefully incomplete) list of references. \cite{baudisch, chiswell, droms1, droms2, droms3, green, charney, wise}. 

Whenever a certain property is preserved under taking both free and direct/tensor products, it is of interest to ask if that property is also preserved under taking graph products.  In many cases, the answer is affirmative--see \cite{valette, antdre, casfim, reck, gpucp}. The purpose of this article is to consider graph products of group actions and graph products of multipliers associated with those actions and to show that the Haagerup property for group actions is stable under graph products.  The notion of the Haagerup property for group actions is relatively new, and at the writing of this article there is neither a free nor a direct product version of this result.  Thus, the main results of this paper are instances
%a satisfying execution 
of the generality of graph products in simultaneously establishing the results for free and direct products of group actions with the Haagerup property.

Readers familiar with Dong-Ruan's article \cite{donrua}, graph products, or both will quickly observe the potential for alternatives to or generalizations of  the combinatorial proof strategy (cf. \S\ref{appb}) of this paper.  Such readers are directed to \S\ref{ob} for a discussion on the obstructions to such approaches.

The paper is organized as follows. In \S\ref{gacp}, we recall the construction of the reduced crossed product $C^*$-algebra associated to the action of a group on a $C^*$-algebra and give the relevant background on Dong-Ruan's definition for the Haagerup property for group actions.  In \S\ref{cohochar} we provide a characterization of the Haagerup property for group actions in terms of 1-cocycles. In \S\ref{gp} we consider graph products of group actions and their corresponding multipliers.  The graph product of group actions is natural enough to define, but to define the graph product of the corresponding multipliers care must be taken to ensure the appropriate commuting relations are satisfied--see Definitions \ref{multcom} and \ref{gppdm}.    In \S\ref{mr}, we give the main results of the paper. The graph product of appropriately commuting positive definite multipliers is again positive definite (Theorem \ref{posdef}). We also establish the analogous result for graph products of positive definite functions on left transformation groupoids. Then we apply Theorem \ref{posdef}  to prove that the graph product of actions possessing the Haagerup property (and whose multipliers suitably commute) has the Haagerup property (Theorem \ref{hp}).  This immediately implies that coarse embeddability of discrete groups into a Hilbert space is stable under graph products (a result originally appearing in \cite{dadgue} for general amalgamated free products).  In \S\ref{ob}, we discuss obstructions to more general or alternative approaches.  \S\ref{appb} is devoted to the proof of Theorem \ref{posdef}. To do this, we establish several results for the kernel of a graph product of multipliers.  These results utilize combinatorics for graph products originally studied by the author in \cite{gpucp}, and thus the arguments  are adapted versions of the arguments in \cite{gpucp}, initially inspired by those in \cite{boca}.

Unless otherwise indicated, all groups in this article are discrete, and all $C^*$-algebras are unital.
%and discuss some preliminaries for graph products and their related combinatorics
%\section{Preliminaries}\label{pre}

\section{The Haagerup property for group actions}\label{gacp} In this section, we first recall some fundamental constructions and facts regarding crossed product $C^*$-algebras; then we will review the Haagerup property for group actions as discussed in \cite{donrua}.  Let $G$ be a group, and let $\mc{A}$ be a unital $C^*$-algebra.  A \emph{group action} $\alpha: G \car \mc{A}$ is a group homomorphism from $G$ into the automorphism group of $\mc{A}$.  Given an action $\alpha: G \car \mc{A}$, we can form the \emph{reduced crossed product $C^*$-algebra} $G \rtimes_{\alpha,r} \mc{A}$ as follows.  Let $\pi: \mc{A} \rightarrow B(\mc{H})$ be a faithful representation of $\mc{A}$, and let $\lambda: G \rightarrow B(\ell_2(G))$ denote the left-regular representation.  We can extend $\pi$ and $\lambda$ to representations $\tilde{\pi}$ and $\tilde{\lambda}$ (respectively) on the Hilbert space $\ell_2\otimes B(\mc{H})$ by putting 
\begin{align*}
\tilde{\lambda} & := \lambda \otimes 1_\mc{H}\\
&\text{and} \\
\tilde{\pi}(a)(\delta_s \otimes \xi) &:= \delta_s \otimes \pi(\alpha_s(a))\xi.
\end{align*}
This gives us a covariant representation $(\tilde{\pi},\tilde{\lambda})$ of the $C^*$-dynamical system $(\mc{A}, G, \alpha)$ on $\ell_2(G) \otimes \mc{H}$; that is, \[\tilde{\pi}(\alpha_s(a)) = \tilde{\lambda}_s\tilde{\pi}(a)\tilde{\lambda}_s^*.\]  When there is no risk of confusion, we will suppress the $\tilde{\pi}$ notation and will write $\lambda_s$ for $\tilde{\lambda}_s$. 

\begin{dfn} 
The \emph{reduced crossed product $C^*$-algebra} $G \rtimes_{\alpha,r} \mc{A}$ is given by the norm-closure of \[C_c(G, \mc{A}):= \left\{\sum_{s \in G} \lambda_s a_s: \text{ finitely many non-zero terms }\right\}\] in $B(\ell_2\otimes \mc{H})$.  
\end{dfn}
\noin The $C^*$-algebra $\mc{A}$ naturally sits inside $G \rtimes_{\alpha,r} \mc{A}$ as a unital subalgebra (elements with all non-identity terms equal to zero), and there exists a faithful conditional expectation onto this copy of $\mc{A}$ given by \[\sum_{s \in G} \lambda_s a_s \mapsto a_e.\]

Group actions are often accompanied by multipliers. A positive definite multiplier is defined as follows.

\begin{dfn}[\cite{ad, donrua}]\label{pdmul}
A map $h: G \rightarrow \mc{Z}(\mc{A})$ is a \emph{positive definite multiplier} if for every $n \in \mb{N}$ and $x_1,\dots, x_n \in G$ we have that the matrix $[\alpha_{x_j}(h_{x_i^{-1}x_j})]_{ij}$ is positive in $\mb{M}_n(\mc{A})$.  Such a map is called \emph{unital} if $h_e = 1_\mc{A}$.
\end{dfn}

\begin{rmk}\label{convention}
When comparing the definitions of the above concept in \cite{ad} and \cite{donrua}, one will notice a discrepancy.  The former demands the positivity of $[\alpha_{x_i}(h_{x_i^{-1}x_j})]_{ij}$, and the latter requires the positivity of $[\alpha_{x_j}(h_{x_i^{-1}x_j})]_{ij}$. Be assured that this is only a difference in convention.  Indeed, given a positive definite multiplier $h$ of one type, one obtains a positive definite multiplier $\tilde{h}$ of the other type by considering $\tilde{h}_s = h_{s^{-1}}^*$.  Unless otherwise indicated, we will follow the convention of Definition \ref{pdmul}.
\end{rmk}

\begin{prop}
Let $\alpha: G \car \mc{A}$ be a group action. If $h: G \rightarrow \mc{Z}(\mc{A})$ is a positive definite multiplier with respect to $\alpha$, then $h_{a^{-1}} ^*= \alpha_a(h_a)$ for every $a \in G$.
\end{prop}

\begin{dfn}
A function $h: G \rightarrow \mc{A}$ \emph{vanishes at infinity} if for any $\varepsilon > 0$, there is a finite subset $F \subset G$ so that $||h_s||< \varepsilon$ for every $s \in G \setminus F$.  The set of all such functions will be denoted $C_0(G, \mc{A})$.
\end{dfn}

\begin{dfn}[\cite{donrua}]
A group action $\alpha: G \car \mc{A}$ has the \emph{Haagerup property} if there exists a sequence of positive definite multipliers $\left\{h_n\right\}$ in $C_0(G, \mc{A})$ such that for every $s \in G, h_{n,s} \rightarrow 1$ as $n \rightarrow \infty$.  That is, $h_n\rightarrow 1_\mc{A}$ pointwise.
\end{dfn}

\begin{prop}\label{unital}
If $\alpha: G \car \mc{A}$ has the Haagerup property, then there exists a sequence of positive definite multipliers  in $C_0(G, \mc{A})$ converging to $1_\mc{A}$ pointwise satisfying the following properties.
\begin{enumerate}
\item $h$ is unital;
 \item $\ds ||h_s|| \leq \frac{1}{2}$ for every $s \in G\setminus \left\{e\right\}$.
\end{enumerate}
\end{prop}

\begin{proof}  Let $h: G \rightarrow \mc{Z}(\mc{A})$ be a positive definite multiplier in $C_0(G,\mc{A})$. Since $||h_s|| \leq ||h_e||$ for every $s \in G$, we may assume without loss of generality that $\ds ||h_s|| \leq \frac{1}{2}$ for every $s \in G$.  Define $\tilde{h}: G \rightarrow \mc{Z}(\mc{A})$ as follows.
\[\tilde{h}_s := \left\{\begin{array}{lcr}
h_s & \text{if} & s \neq e\\
1_\mc{A} & \text{if} & s = e
\end{array}\right.
\]
Let $x_1\cdots x_n$ be a sequence of elements in $G$.  We have that \[ [\alpha_{x_j}(\tilde{h}_{x_i^{-1}x_j})]_{ij} = [\alpha_{x_j}(h_{x_i^{-1}x_j})]_{ij} + [A_{ij}]_{ij}\] where 
\[A_{ij} = \left\{ \begin{array}{lcr}
0 & \text{if} & x_i \neq x_j\\
\alpha_{x_i}(1-h_e) & \text{if} & x_i = x_j.
\end{array}\right.\]
We have that $0 \leq h_e \leq 1_\mc{A}$, and so it is a direct computation to see that $[A_{ij}]_{ij}$ is positive.  Then given a sequence $\left\{h_n\right\} \subset C_0(G, \mc{A})$ of positive definite multipliers converging to 1 pointwise, we have that $\left\{\tilde{h}_n\right\} \subset C_0(G, \mc{A})$ is a sequence of positive definite multipliers with the desired properties converging to 1 pointwise.
\end{proof}

\section{Cohomological characterization}\label{cohochar}

In this section, we record a cohomological characterization of the Haagerup property for a group action.  Experts will observe that this result can be deduced directly from the approach using so-called $\alpha$-negative definite functions in B\'edos-Conti's paper \cite{bedcon}.  In order to illustrate the analogy between positive definite functions for groups and positive definite multipliers for group actions we present this characterization with an approach using positive definite multipliers.  

Let $G$ be a group, $\mc{A}$ a $C^*$-algebra, $\alpha: G \car \mc{A}$ an action, and $h: G \rightarrow \mc{Z}(\mc{A})$ a corresponding unital positive definite multiplier. Only in this subsection will we follow the convention in \cite{ad} for positive definite multipliers; that is, $[\alpha_{x_i}(h_{x_i^{-1}x_j})]_{ij} \geq 0$. Due to Remark \ref{convention}, there is no loss of generality. Let $X$ denote a Hilbert $\mc{A}$-module, and let $\mc{I}(X)$ denote the group of bijective $\mb{C}$-linear isometries $u: X \rightarrow X$.

\begin{dfn}[\cite{combes, ad, bedcon}]
An \emph{$\alpha$-equivariant action} $u$ of $G$ on $X$ is a map $u: G \rightarrow \mc{I}(X)$ satisfying the following conditions.
\begin{enumerate}
\item $\alpha_s(\langle x | y \rangle) = \langle u_sx|u_sy\rangle$ for every $s \in G, x,y \in X$
\item $u_s(x\cdot a) = (u_s x)\cdot \alpha_s(a)$ for every $s \in G, x \in X, a \in \mc{A}$
\end{enumerate}
\end{dfn}

\begin{dfn}[\cite{ad, bedcon}]
Given an $\alpha$-equivariant action $u: G \rightarrow \mc{I}(X)$, a map $b: G \rightarrow X$ is called a \emph{1-cocycle} with respect to $u$ if \[b(st) = b(s) + u_s(b(h))\] for every $s,t \in G$.  A 1-cocycle is called \emph{central} if $\langle b(s) | b(s) \rangle \in \mc{Z}(\mc{A})$ for every $s \in G$. Note that this differs from the notion of 1-cocycles for group actions considered in \cite{contak, popabern}.
\end{dfn}

\begin{prop}[\cite{ad}]\label{gns}
With $G, \mc{A}, \alpha, h$ as above, there exist an $\alpha$-equivariant action $u$ on a Hilbert $\mc{A}$-module $X$ and a vector $\xi$ such that $h(s) = \langle u_t \xi |\xi\rangle$.  Thus, we obtain a central 1-cocycle with respect to $u$ given as $b: s \mapsto \xi - u_s\xi$.
\end{prop}

\noin  Such a Hilbert $\mc{A}$-module and $\alpha$-equivariant action is obtained in a GNS fashion.  In particular we consider the following $\mc{A}$-valued positive semi-definite sesquilinear form on $C_c(G,\mc{A})$ induced by $h$:  \[\langle f | g \rangle_h  = \sum_{s,t \in G} g(s)^*\alpha_s(h_{s^{-1}t})f(t),\] and form the corresponding Hilbert $\mc{A}$-module.

\begin{dfn}[\cite{bedcon}]
A function $c: G \rightarrow \mc{A}^+$ is \emph{spectrally proper} if for every $s \in G,$ $s \mapsto \inf (\sigma (c(s)))$ is proper where $\sigma(T)$ denotes the spectrum of $T$.  That is, for any $R>0$, the set $\left\{ s \in G: \inf(\sigma(c(s))) \leq R\right\}$ is finite.  A 1-cocycle $b: G \rightarrow X$ is \emph{spectrally proper} if $s\mapsto \langle b(s) | b(s)\rangle$ is spectrally proper.
\end{dfn}

We now give the following necessary condition for an action to have the Haagerup property. The argument is parallel to the argument for the group case in Theorem 12.2.4 in \cite{brownozawa}; for the sake of completeness we include the proof.  

\begin{prop}
If an action $\alpha: G \car \mc{A}$ has the Haagerup property then it admits a central spectrally proper 1-cocycle.
\end{prop}

\begin{proof}
Let $h_n$ be a sequence of unital positive definite multipliers in $C_0(G, \mc{A})$ converging to $1_\mc{A}$ pointwise. By Lemma 2.6 of \cite{bedcon}, we may assume that $h_{n,s}\geq 0$ for every $s \in G$. Enumerate $G = \left\{s_k\right\}$. Suppose that for every $k \in \mb{N}$, we have $||1 - h_{n,s_k}|| < 2^{-n}$ for every $n \geq k$.  For each $n$, by Proposition \ref{gns} there is an $\alpha$-equivariant action $u_n$ on a Hilbert $\mc{A}$-module $X_n$ and a vector $\xi_n \in X_n$ such that $h_{n,s} = \langle u_{n,s}\xi_n|\xi_n\rangle$.  Let $X = \oplus X_n$ and $u = \oplus u_n$.  Then $u$ is an $\alpha$-equivariant action and we obtain a central 1-cocycle $b: G \rightarrow X$ given by $b(s) = (\xi_n - u_{n,s}\xi_n)_n, s \in G$.  Then we have \[\langle b(s) | b(s) \rangle = \sum_{n=1}^\infty \langle \xi_n - u_{n,s}\xi_n|\xi_n - u_{n,s}\xi_n\rangle = \sum_{n=1}^\infty 2(1 - h_{n,s}),\]  and thus $b$ is spectrally proper.  Indeed, since $\langle b(s) | b(s) \rangle \in \mc{Z}(\mc{A})^+$, if for some $s \in G$, 
\begin{align}
\inf(\sigma (\langle b(s) | b(s) \rangle)) \leq N,\label{sp}
\end{align} then $||h_{n,s}|| \geq \frac{1}{2}$ for some $n \in \left\{1, \dots, N\right\}$; otherwise, $\langle b(s) | b(s) \rangle > N 1_\mc{A}$, contradicting \eqref{sp}. So if $b$ is not spectrally proper (i.e. \eqref{sp} occurs for some $N \in \mb{N}$ and infinitely many $s \in G$), then it follows that for some $n \in \left\{1, \dots, N\right\}$, $h_n \notin C_0(G, \mc{A})$--absurd!
\end{proof}

The following is a consequence of B\'edos-Conti's $C^*$-dynamical version of Schoenberg's theorem obtained in \cite{bedcon}.

\begin{prop}[\cite{bedcon}]\label{schoen}
If $b: G \rightarrow X$ is a 1-cocycle with respect to an $\alpha$-equivariant action $u$ such that $\langle b(s) | b(s) \rangle \in \mc{Z}(\mc{A})$ for every $s \in G$, then $s \mapsto \text{exp}(-t \langle b(s)|b(s)\rangle^2)$ is a positive definite multiplier with respect to $\alpha$ for every $t >0$.
\end{prop}

\begin{dfn}[\cite{bedcon}]
A function $\psi: G \rightarrow \mc{A}$ is \emph{$\alpha$-negative definite} if \[\alpha_s(\psi(s^{-1})) = \psi(s)^*\] for every $s \in G$, and for any $n \in \mb{N}, s_1,\dots, s_n \in G$, and $b_1,\dots, b_n \in \mc{A}$ with $\sum_{i=1}^n b_i = 0$, we have \[\sum_{i,j=1}^n b_i^* \alpha_{g_i}(\psi(g_i^{-1}g_j))b_j \leq 0.\]  We say $\psi$ is \emph{normalized} if $\psi(e) = 0$.
\end{dfn}

We can now characterize the Haagerup property for a group action $\alpha: G \car \mc{A}$ as follows.

\begin{thm}
Let $\alpha: G \car \mc{A}$ be an action.  The following are equivalent.

\begin{enumerate}
\item The action $\alpha: G \car \mc{A}$ has the Haagerup property;
\item (\cite{bedcon}) The action $\alpha: G \car \mc{A}$ admits a spectrally proper $\mc{Z}(\mc{A})^+$-valued normalized $\alpha$-negative definite function on $G$.
\item The action $\alpha: G \car \mc{A}$ admits a spectrally proper 1-cocycle.
\end{enumerate}
\end{thm}

\begin{proof}
It remains to show (3) $\Rightarrow$ (1).  By Proposition \ref{schoen}, if $b: G \rightarrow X$ is a spectrally proper 1-cocycle, then \[h_n : s \mapsto \text{exp}\Big(-\frac{\langle b(s)|b(s)\rangle^2}{n}\Big)\] is a sequence of positive definite multipliers in $C_0(G, \mc{A})$ converging to $1_\mc{A}$ pointwise.
\end{proof}

\section{Graph products of group actions and multipliers}\label{gp}

In this section, we discuss some preliminaries regarding graph products.  We then consider graph products of group actions and graph products of the corresponding multipliers.  

Fix a simplicial (i.e. undirected, no single-vertex loops, at most one edge between vertices) graph $\Gamma = (V, E)$, where $V$ denotes the set of vertices of $\Gamma$ and $E \subset V \times V$ denotes the set of edges of $\Gamma$. Given discrete groups $\left\{G_v\right\}_{v \in V}$ one can define the graph product of the $G_v$'s as follows.

\begin{dfn}[\cite{green, casfim}]
The graph product $\bigstar_\Gamma G_v$ is given by the free product $* G_v$ modulo the relations $[g,h] =1$ whenever $g \in G_v, h \in G_w$ and $(v,w) \in E$.
\end{dfn}

\begin{exmpl}[\textbf{Complete multipartite graphs}]
Let $n_1, \dots, n_k \in \mb{N}$, and let $K_{n_1,\dots, n_k}$ denote the complete $k$-partite graph with $n_j$ vertices in the $j^\text{th}$ independent set.  For instance, the following is the graph $K_{1,2,3}$.  

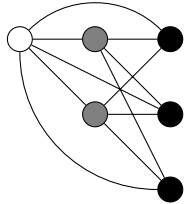
\begin{center}
\begin{tikzpicture}
\path (0,2) node [draw,shape=circle,] (p0) {}
(1,2) node [draw,shape=circle,fill=gray]  (p1) {}
(1,1) node [draw,shape=circle,fill=gray] (p2) {}
(2,2) node [draw,shape=circle,fill=black] (p3) {}
(2,1) node [draw,shape=circle,fill=black] (p4) {}
(2,0) node [draw,shape=circle,fill=black] (p5) {};
\draw (p0) -- (p1)
(p0) -- (p2)
(p0) to [out=45,in=135] (p3)
(p0) -- (p4)
(p0) to [out=-90,in=-180] (p5)
(p1) -- (p3)
(p1) -- (p4)
(p1) -- (p5)
(p2) -- (p3)
(p2) -- (p4)
(p2) -- (p5);
\end{tikzpicture}
\end{center}
Let \[V = \left\{v_{1,1},\dots, v_{1, n_1}, \dots, v_{i,1}, \dots, v_{i,n_i},\dots, v_{k,1},\dots, v_{k, n_k}\right\}\] be the vertex set of $K_{n_1,\dots, n_k}$.  For $1 \leq i \leq k$ and $1 \leq j \leq n_i$, let $G_{v_{i,j}}$ be a group.  Then \[\bigstar_{K_{n_1,\dots,n_k}} G_{v_{i,j}} \cong \prod_{i=1}^k (*_{j=1}^{n_i} G_{v_{i,j}}).\] 
\end{exmpl}

When working with graph products, the bookkeeping can be done by considering words with letters from the vertex set $V$.  Such words are given by finite sequences of elements from $V$ and will be denoted with bold letters.  In order to encode the commuting relations given by $\Gamma$, we consider the equivalence relation generated by the following relations.
\begin{align*}
(v_1,\dots, v_i, v_{i+1},\dots, v_n) &\sim (v_1,\dots, v_i, v_{i+2}, \dots, v_n) &\text{if} &&v_i = v_{i+1}\\
(v_1,\dots, v_i, v_{i+1},\dots, v_n) &\sim (v_1,\dots, v_{i+1}, v_i,\dots, v_n) &\text{if} &&(i,i+1) \in E.
\end{align*}
The concept of a reduced word is central to the theory of graph products.  The following definition is Definition 3.2 of \cite{spenic} in graph language; the equivalent definition in \cite{casfim} appears differently.

\begin{dfn}\label{redv}
A word $\mbf{v} = (v_1,\dots,v_n)$ is \emph{reduced} if whenever $v_k = v_l, k < l$, then there exists a $p$ with $k< p < l$ such that $(v_k, v_p) \notin E$.  Let $\mc{W}_\text{red}$ denote the set of all reduced words.  We take the convention that the empty word is reduced.
\end{dfn}

\begin{prop}[\cite{green,casfim}]\label{reducedlemma}\hspace*{\fill}
\begin{enumerate}
\item Every word $\mbf{v}$ is equivalent to a reduced word $\mbf{w} = (w_1,\dots, w_n)$.  (We let $|\mbf{w}|=n$ denote the \emph{length} of the reduced word.)
\item If $\mbf{v} \sim \mbf{w}\sim\mbf{w}'$ with both $\mbf{w}$ and $\mbf{w}'$ reduced, then the lengths of $\mbf{w}$ and $\mbf{w}'$ are equal and $\mbf{w}' = (w_{\sigma(1)},\dots,w_{\sigma(n)})$ is a permutation of $\mbf{w}$.  Furthermore, this permutation $\sigma$ is unique if we insist that whenever $w_k = w_l, k< l$ then $\sigma(k) < \sigma(l)$.
\end{enumerate}
\end{prop}

\begin{dfn}
A \emph{reduced word} $x \in \bigstar_\Gamma G_v$ is an element of the form $x = x_1\cdots x_m$ where $x_k \in G_{v_k}$ and $(v_1,\dots,v_m) \in \mc{W}_\text{red}$.  In such an instance we write $(v_1,\dots, v_m) = \mbf{v}_x$ and say $|x| = m$--denoting the \emph{length} of $x$ (well-defined by Proposition \ref{reducedlemma}).  Accepting the common risks of abusing notation, we let $\mc{W}_\text{red}$ also denote the set of reduced words in $\bigstar_\Gamma G_v$.  We will take the convention that the identity element of $\bigstar_\Gamma G_v$ is reduced and has length zero.
\end{dfn}

%\begin{dfn}
%Given a reduced word $s = s_1\cdots s_n \in \bigstar_\Gamma G_v$, let $P(s) \subset S_n$ be the subset of permutations $\sigma$ for which $s = s_{\sigma(1)} \cdots s_{\sigma(n)}$.
%\end{dfn}

%\section{Graph products of group actions and commuting multipliers}\label{mult}

Fix a simplicial graph $\Gamma = (V, E)$ and a unital separable $C^*$-algebra $\mc{A}$.  For each $v \in V$, let $G_v$ be a discrete group with action $\alpha_v: G_v \car \mc{A}$.  We can view each group action as a group homomorphism into the automorphism group of $\mc{A}$; i.e. $\alpha_v: G_v \rightarrow \text{Aut}(\mc{A}).$  If whenever $(v,w) \in E$ we have $\alpha_v(g)\alpha_w(h) = \alpha_w(h)\alpha_v(g)$ in $\text{Aut}(\mc{A})$ for $g \in G_v, h \in G_w$, then by the universal property of graph products of groups, we can form the \emph{graph product of the actions $\alpha_v: G_v \car \mc{A}$}, denoted $\bigstar_\Gamma \alpha_v: \bigstar_\Gamma G_v \car \mc{A}$.  When the actions $\left\{\alpha_v\right\}$ satisfy the commuting condition above, we will say that they \emph{commute according to $\Gamma$}.

%\begin{exmpl}
%In the case that $\Gamma$ has no edges, we obtain a free product of actions on $\mc{A}$.  And we have \[ (* G_v) \ltimes_{*\alpha_v} \mc{A} \cong *_{\mc{A}} (G_v \ltimes_{\alpha_v} \mc{A})\] where $*_\mc{A}$ denotes the full amalgamated free product over the canonical copy of $\mc{A}$ in the full crossed product $G_v \ltimes_{\alpha_v} \mc{A}$.
%\end{exmpl}

It is clear enough to describe commuting group actions, in order to define the graph product of the corresponding multipliers, we need to know what it means for two multipliers to commute.

\begin{dfn}\label{multcom}
Let $G_1, G_2$ be two discrete groups and $\mc{A}$ be a unital $C^*$-algebra.  Let $\alpha_i: G_i \car \mc{A}, i = 1,2$ be commuting actions, and let $h_i: G_i \rightarrow \mc{Z}(\mc{A}), i=1,2$ be unital positive definite multipliers with respect to $\alpha_i, i = 1,2$ respectively.  We say that $h_1$ and $h_2$ \emph{commute} if \[\alpha_{i,a}(h_{j,b}) = h_{j,b}\] for $a \in G_i, b \in G_j, i, j \in \left\{1,2\right\}, i \neq j$.
%the following two conditions are satisfied for every $a \in G_1\setminus\left\{e\right\}$ and  $b \in G_2\setminus \left\{e\right\}$.
%\begin{enumerate}
%\item For every multiplicative linear functional $\omega \in \mc{Z}(\mc{A})^*$, \[|\omega(\alpha_{1,a}(h_{2,b}))| = |\omega(h_{2,b})|.\]
%\item \hspace*{\fill}
%\begin{align}
%\alpha_{2,b^{-1}}(h_{1,a}) h_{2,b} &= \alpha_{1,a^{-1}}(h_{2,b})h_{1,a}.\label{comrel}
%\end{align} 
%\end{enumerate} 
%See Remark \ref{commot} for the motivation for this definition.  
Given a simplicial graph $\Gamma = (V,E)$, groups $\left\{G_v\right\}_{v \in V}$, actions $\left\{\alpha_v: G_v \car \mc{A}\right\}_{v \in V}$ that commute according to $\Gamma$, and respective unital positive definite multipliers $\left\{h_v: G_v \rightarrow \mc{Z}(\mc{A})\right\}_{v \in V}$, we say that the multipliers \emph{commute according to $\Gamma$} if $h_v$ and $h_w$ commute whenever $(v,w) \in E$.
\end{dfn}

%tensor product example

\begin{exmpl}\label{tensorexmpl}
Let $G_1,G_2$ be two discrete groups and $\mc{A}_1, \mc{A}_2$ be two unital $C^*$-algebras.  Let $\alpha_i: G_i \car \mc{A}_i, i=1,2$ be actions, and let $h_i: G_i \rightarrow \mc{Z}(\mc{A}_i), i =1,2$ be unital positive definite multipliers with respect to $\alpha_i, i=1,2$ respectively.  Consider the actions $\alpha_1 \otimes \text{id}: G_1 \car \mc{A}_1 \otimes \mc{A}_2$ and $\text{id} \otimes \alpha_2: G_2 \car \mc{A}_1\otimes \mc{A}_2$ (for whichever tensor closure).  Then the multipliers $h_1 \otimes 1$ and $1 \otimes h_2$ commute.
\end{exmpl}

The next example shows that given commuting amenable actions with (not necessarily commuting) positive definite multipliers, one can construct commuting actions with commuting multipliers.  Some preliminaries are in order before presenting the example.  Recall that an action $\alpha$ of a group $G$ on a compact Hausdorff space $X$ is \emph{amenable} if there exists a sequence (all groups are countable discrete) of continuous maps $m_i: X \rightarrow \text{Prob}(G)$ such that for each $\ds g \in G, \lim_{i\rightarrow \infty} \Big( \sup_{x \in X} ||g.m_i^x - m_i^{\alpha_g(x)}||_1\Big) = 0$ where $g.m_i^x(h) = m_i^x(g^{-1}h)$ and $\text{Prob}(G)\subset \ell^1(G)$ denotes the space of probability measures on $G$ (cf. \cite{brownozawa}). It is well-known that there is a 1-1 correspondence between actions of a group $G$ on a compact Hausdorff space $X$ and the actions of $G$ on $C(X)$.  Given $\alpha: G \car X$ we obtain $\hat{\alpha}: G \car C(X)$ by setting $\hat{\alpha}_s(f)(x) = f(\alpha_{s^{-1}}(x))$ for every $f \in C(X), x \in X$.  Lastly, we recall ultraproduct constructions for $C^*$-algebras.  Let $I$ be an indexing set and let $\mc{U}$ be an ultrafilter on $I$ (see Appendix A of \cite{brownozawa}).  For each $i \in I$, let $\mc{A}_i$ be a unital $C^*$-algebra.  Let $\prod_I \mc{A}_i$ denote the $\ell^\infty$-direct sum of the $\mc{A}_i$'s and let $\ds N_\mc{U} = \left\{ (a_i) \in \prod_I \mc{A}_i | \lim_{i \rightarrow \mc{U}} ||a_i|| = 0\right\}$.  Then the \emph{ultraproduct $C^*$-algebra} $\prod_\mc{U} \mc{A}_i$ is given by \[\prod_\mc{U} \mc{A}_i = (\prod_I \mc{A}_i)/N_\mc{U}.\]  If $\mc{A}_i = \mc{A}$ for every $i \in I$, we write $\prod_\mc{U} \mc{A} = \mc{A}_\mc{U}$ and call it the \emph{ultrapower} of $\mc{A}$.

\begin{exmpl}\label{ultraexample}
Let $G, G'$ be two discrete groups acting on a compact Hausdorff space $X$ via amenable actions $\alpha$ and $\alpha'$ respectively.  Let $h: G \rightarrow C(X)$ and $h': G' \rightarrow C(X)$ be positive definite multipliers with respect to the respective induced actions on $C(X)$.  Since $\alpha$ is amenable, there exists a sequence of continuous maps $m_i: X \rightarrow \text{Prob}(G)$ such that for every $\ds g \in G, \lim_{i\rightarrow \infty} \Big( \sup_{x \in X} ||g.m_i^x - m_i^{\alpha_g(x)}||_1\Big) = 0$.  Consider the positive definite multiplier on $h'_i: G' \rightarrow C(X)$ given by \[(h'_i)_s(x) = \sum_{g \in G} m_i^x(g) h'_s(\alpha_{g^{-1}}(x))\] for $s \in G'$. Let $\mc{U}$ be a free ultrafilter on $\mb{N}$ and consider $(h'_i)_\mc{U}: G' \rightarrow C(X)_\mc{U}$  given by $((h'_i)_\mc{U})_s = ((h'_i)_s)_\mc{U}$.  Clearly $(h'_i)_\mc{U}$ is a positive definite multiplier with respect to the \tql diagonal\tqr action $\hat{\alpha}'_\mc{U}$ given by $\hat{\alpha}'_\mc{U}((f_i)_\mc{U}) = (\hat{\alpha}'(f_i))_\mc{U}$.  We claim that $(h'_i)_\mc{U}$ is invariant under the similarly defined diagonal action $\hat{\alpha}_\mc{U}$.  Fix $a \in G, s \in G', x \in X,$ and $i \in \mb{N}$. Then we have
\begin{align*}
& \Big|\hat{\alpha}_a(h'_i)_s(x) - (h'_i)_s(x)\Big| \\
&= \Big|(h'_i)_s(\alpha_{a^{-1}}(x)) - (h'_i)_s(x)\Big| \\
& = \Big|\sum_{g \in G} m_i^{\alpha_{a^{-1}}(x)}(g) h'_s(\alpha_{g^{-1}}(\alpha_{a^{-1}}(x))) - \sum_{g \in G} m_i^x(g) h'_s(\alpha_{g^{-1}}(x))\Big|\\
& = \Big|\sum_{g \in G} m_i^{\alpha_{a^{-1}}(x)}(g) h'_s(\alpha_{g^{-1}a^{-1}}(x)) - \sum_{g \in G} m_i^x(ag) h'_s(\alpha_{g^{-1}a^{-1}}(x))\Big|\\
& = \Big|\sum_{g \in G} m_i^{\alpha_{a^{-1}}(x)}(g) h'_s(\alpha_{g^{-1}a^{-1}}(x)) - \sum_{g \in G} (a^{-1}).m_i^x(g) h'_s(\alpha_{g^{-1}a^{-1}}(x))\Big|\\
& = \Big|\sum_{g \in G} (m_i^{\alpha_{a^{-1}}(x)}(g) -(a^{-1}).m_i^x(g)) h'_s(\alpha_{g^{-1}a^{-1}}(x))\Big|\\
& \leq \sum_{g \in G} \Big|m_i^{\alpha_{a^{-1}}(x)}(g) -(a^{-1}).m_i^x(g)\Big| \cdot ||h'_s||\\
& \leq \Big(\sup_{x \in X} \Big|\Big|m_i^{\alpha_{a^{-1}}(x)}(g) -(a^{-1}).m_i^x(g)\Big|\Big|_1\Big)\cdot ||h'_s||
\end{align*}
Therefore, $((\hat{\alpha})_\mc{U})_a(((h'_i)_\mc{U})_s) = ((h'_i)_\mc{U})_s$ for every $a \in G, s \in G'$.  By a symmetric argument we get that $(h_i)_\mc{U}$ is invariant under the diagonal action $\hat{\alpha}'_\mc{U}$.  Thus, the actions $\hat{\alpha}_\mc{U}$ and $\hat{\alpha}'_\mc{U}$ commute and the corresponding positive definite multipliers $(h_i)_\mc{U}$ and $(h'_i)_\mc{U}$ commute. 
\end{exmpl}

\noin Note that in the crossed product algebra $(\bigstar_\Gamma G_v) \rtimes_{r, \bigstar_\Gamma \alpha_v} \mc{A}$, if $\left\{h_v: G_v \rightarrow \mc{Z}(\mc{A})\right\}_{v \in V}$ is a collection of unital positive definite multipliers that commute according to $\Gamma$, then whenever $(v,w) \in E$ and $a \in G_v, b \in G_w$, we have \[\lambda_{v,a} h_{w,b} = h_{w,b} \lambda_{v,a}.\]
We are now ready to define the graph product of positive definite multipliers.

\begin{dfn}\label{gppdm}
Let $\mc{A}$ be a unital $C^*$-algebra, and fix a simplicial graph $\Gamma= (V, E)$.  For each $v \in V$, let $G_v$ be a group, and let $\alpha_v: G_v \car \mc{A}$ be an action. For each $v \in V$, let $h_v: G_v \rightarrow \mc{Z}(\mc{A})$ be a unital positive definite multiplier with respect to the action $\alpha_v: G_v \car \mc{A}$.  Suppose that the actions $\left\{\alpha_v\right\}_{v\in V}$ and positive definite multipliers $\left\{h_v\right\}_{v \in V}$ commute according to $\Gamma$. Let $\bigstar_\Gamma h_v: \bigstar_\Gamma G_v \rightarrow \mc{Z}(\mc{A})$ denote the graph product of the positive definite multipliers defined as follows. Given a reduced word $s = s_1\cdots s_n \in \bigstar_\Gamma G_v$ with $s_j \in G_{v_j}$ for $1 \leq j \leq n$, put \[(\bigstar_\Gamma h_v)_s = (\bigstar_\Gamma h_v)_{s_1\cdots s_n} = \alpha_{p_1}^{-1}(h_{v_1,s_1}) \cdots \alpha_{p_{n-1}}^{-1}(h_{v_{n-1},s_{n-1}}) h_{v_n, s_n}\] where $p_j = s_{j+1}\cdots s_{n}$ for $1 \leq j \leq n-1$ and $\alpha$ denotes the graph product action $\bigstar_\Gamma \alpha_v$. 
\end{dfn}

\begin{exmpl}\label{fex}
To illustrate where this definition comes from, consider the free product case (i.e. $\Gamma$ has no edges).  For each $v \in V,$ the multiplier $h_v$ corresponds with a bounded $\mc{A}$-bimodule map $\Phi_v: G_v \ltimes_{\alpha,r} \mc{A} \rightarrow G_v \ltimes_{\alpha,r} \mc{A}$ such that $\Phi_v(\lambda_s) = \lambda_sh_v(s)$ for $s \in G_v$.  Thus, the free product multiplier $* h_v$ should correspond with the amalgamated free product $\mc{A}$-bimodule map $*_\mc{A} \Phi_v: (*G_v) \ltimes_{*\alpha_v, r} \mc{A} \rightarrow (*G_v) \ltimes_{*\alpha_v, r} \mc{A}$ so that, for $1 \leq i \leq n, s_i \in G_{v_i}, v_i \neq v_{i+1}$,  
\begin{align*}
*_\mc{A} \Phi_v (\lambda_{s_1\cdots s_n}) &= *_\mc{A} \Phi_v (\lambda_{s_1}\cdots \lambda_{s_n})\\
&=\Phi_{v_1}(\lambda_{s_1}) \cdots \Phi_{v_n}(\lambda_{s_n})\\
& = \lambda_{s_1}h_{v_1,s_1} \cdots \lambda_{s_n}h_{v_n,s_n}\\
& = \lambda_{s_1}\cdots \lambda_{s_n} \alpha_{p_1}^{-1}(h_{v_1,s_1}) \cdots \alpha_{p_{n-1}}^{-1}(h_{v_{n-1},s_{n-1}})h_{v_n,s_n}
\end{align*} where $p_j = s_{j+1}\cdots s_n$ for $1 \leq j \leq n-1$.  
\end{exmpl}

To justify the presence of the commuting condition for multipliers in the above construction, consider the following example.

\begin{exmpl}\hspace*{\fill}
\begin{enumerate}
\item Let $G_1, G_2$ be two discrete groups, and let $\mc{A}$ be a unital $C^*$-algebra on which $G_1$ and $G_2$ act via $\alpha_i: G_i \car \mc{A}, i = 1,2$. Suppose that $\alpha_1$ and $\alpha_2$ commute as described above.  Let $h_i: G_i \rightarrow \mc{Z}(\mc{A})$ be a unital positive definite multiplier for $i = 1,2$ where $h_2(g) = 1_\mc{A}$ for every $g \in G_2$.  We wish to form a positive definite multiplier on $G_1\times G_2$ from the component multipliers $h_1$ and $h_2$.  Following the example of the free product, for $g_i \in G_i, i = 1,2$ we consider the following definition. \[(h_1\times h_2)_{(g_1,g_2)} := \alpha_{2,g_2}^{-1}(h_{1, g_1})1_\mc{A}\]  But since $(g_1,e)$ and $(e, g_2)$ commute, the necessity of a well-defined direct product of multipliers demands that simultaneously, we have \[(h_1\times h_2)_{(g_1,g_2)} := \alpha_{1,g_1}^{-1}(1_\mc{A})h_{1,g_1} = h_{1,g_1}.\]  Thus, from this perspective we must have that $\alpha_{2,g_2}(h_{1,g_1}) = h_{1,g_1}$ for all $g_i \in G_i, i =1,2$.

\item Taking an alternative approach, we might consider simply taking a pointwise product of multipliers (cf. Lemma 2.6 of \cite{bedcon}) with no regard for the actions.  Consider the same groups $G_1, G_2$, $C^*$-algebra $\mc{A}$, actions $\alpha_1, \alpha_2$, and multipliers $h_1$, $h_2 \equiv 1_\mc{A}$ as in part (1) of this example.  If we define the direct product of multipliers $(h_1\times h_2)$ as \[(h_1\times h_2)_{(g_1,g_2)}:= h_{1,g_1}h_{2,g_2},\] then there is the expectation that the matrix \[\left[\begin{matrix} 1_\mc{A} & \alpha_{2,g_2}(h_{1,g_1^{-1}})\\ \alpha_{1,g_1}(h_{1,g_1}) & 1_\mc{A}\end{matrix}\right]\] is positive.  Thus it is necessary that $ \alpha_{2,g_2}(h_{1,g_1^{-1}}) = \alpha_{1,g_1}(h_{1,g_1}^*) = h_{1,g_1^{-1}}$, and once again we are required to assume that $\alpha_{2,g_2}(h_{1,g_1}) = h_{1,g_1}$ for all $g_i \in G_i, i =1,2$. Once this assumption is in place, the pointwise product of the two (commuting) multipliers is exactly the product described in Definition \ref{gppdm} for this case.
\end{enumerate}
\end{exmpl}

%\begin{rmk}\label{commot}
%With Definition \ref{gppdm} in place, we can now explain the definition of commuting positive definite multipliers.  Suppose $G_1$ and $G_2$ are groups acting on a $C^*$-algebra $\mc{A}$ with commuting actions $\alpha_i: G_i \car \mc{A}, i = 1,2$ and let $h_i: G_i \rightarrow \mc{Z}(\mc{A}), i = 1,2$ be the corresponding unital positive definite multipliers.  We can form the direct product of the actions $\alpha_1\times \alpha_2: G_1\times G_2 \car \mc{A}$, and we would also like to form the \emph{well-defined} direct product of the positive definite multipliers $h_1\times h_2: G_1\times G_2 \rightarrow \mc{Z}(\mc{A})$.  Let $a \in G_1$ and $b \in G_2$; we insist that $(h_1\times h_2)_{ab} = (h_1\times h_2)_{ba}$.  Following the formula from Example \ref{fex}, we obtain
%\begin{align*}
%(h_1\times h_2)_{ab} &= (\alpha_1\times \alpha_2)_{b^{-1}}(h_{1,a})h_{2,b}\\
%&\text{and}\\
%(h_1\times h_2)_{ba} &= (\alpha_1\times \alpha_1)_{a^{-1}}(h_{2,b})h_{1,a}.
%\end{align*}
%Thus we must have \[(\alpha_1\times \alpha_2)_{b^{-1}}(h_{1,a})h_{2,b} = (\alpha_1\times \alpha_1)_{a^{-1}}(h_{2,b})h_{1,a}.\] By applying this reasoning to a direct inductive argument, we obtain the following fact.
%\end{rmk}

\begin{rmk}
Suppose $\left\{\alpha_v: G_v \car \mc{A}\right\}$ and $\left\{h_v: G_v \rightarrow \mc{Z}(\mc{A})\right\}$ are actions and unital positive definite multipliers (respectively) that commute according to the simplicial graph $\Gamma = (V,E)$.  We see that if $(w,w') \in E$ and $a \in G_w, b \in G_{w'}$, then we have \[(\bigstar_\Gamma h_v)_{ab} = h_{w,a}h_{w',b} = h_{w',b}h_{w,a} = (\bigstar_\Gamma h_v)_{ba}.\]  So by applying this reasoning to a direct inductive argument, we obtain the following fact.
\end{rmk}
\begin{prop}
The map $\bigstar_\Gamma h_v$ is well-defined.
\end{prop}

\section{Results}\label{mr}

%We open this section with the following theorem.  
The first theorem of this section is a $C^*$-dynamical system version of the main result of \cite{gpucp}.  The proof here proceeds mostly mutatis mutandis to the argument in \cite{gpucp}. For the sake of exposition and completeness, we have included the proof of Theorem \ref{posdef} in \S \ref{appb}.
\begin{thm}\label{posdef}
Fix a simplicial graph $\Gamma = (V,E)$ and a unital $C^*$-algebra $\mc{A}$.  For each $v \in V$, let $G_v$ be a group, $\alpha_v: G_v \car \mc{A}$ be an action, and $h_v: G_v \rightarrow \mc{Z}(\mc{A})$ be a corresponding unital positive definite multiplier.  Suppose that $\left\{\alpha_v\right\}_{v \in V}$ and $\left\{h_v \right\}_{v \in V}$ commute with respect to $\Gamma$.  Then $\bigstar_\Gamma h_v$ is a positive definite multiplier for the action $\bigstar_\Gamma \alpha_v: \bigstar_\Gamma G_v \car \mc{A}$.
\end{thm}

In \S4 of \cite{donrua}, Dong-Ruan discuss the Haagerup property for left transformation groupoids.  In particular, given a group $G$, a compact Hausdorff space $X$, and an action $\alpha: G \car X$, one can form the left transformation groupoid $G \rtimes X$ (cf. \cite{donrua}). There is a 1-1 correspondence between positive definite multipliers $h: G \rightarrow C(X)$ and positive definite functions $\tilde{h}: G \rtimes X \rightarrow \mb{C}$ (cf. \cite{tu}) given by $\tilde{h}(s,x) = h_s(x), s \in G, x \in X$. Given a compact Hausdorff space $X$, groups $\left\{G_v\right\}_{v\in V}$, and actions $\left\{\alpha_v: G_v \car X\right\}_{v \in V}$ (or equivalently $\alpha_v: G \car C(X)$) that commute according to the simplicial graph $\Gamma = (V,E)$, then we can form the graph product transformation groupoid $(\bigstar_\Gamma G_v) \rtimes X$.   We say that the collection of unital positive definite functions $\left\{\tilde{h}: G_v \rtimes X \rightarrow \mb{C}\right\}_{v \in V}$ \emph{commute according to $\Gamma$} if the corresponding unital positive definite multipliers $\left\{h_v: G_v \rightarrow C(X)\right\}_{v \in V}$ commute according to $\Gamma$. In case $\left\{\alpha_v\right\}_{v\in V}$ and $\left\{\tilde{h}_v\right\}_{v\in V}$ commute according to $\Gamma$, we can define $\bigstar_\Gamma \tilde{h}_v: \bigstar_\Gamma G_v \rtimes X \rightarrow \mb{C}$ by \[(\bigstar_\Gamma \tilde{h}_v)(s,x) = (\bigstar_\Gamma h_v)_s(x)\] for $s\in \bigstar_\Gamma G_v$ and $x \in X$. In this setting, we have the following corollary to Theorem \ref{posdef}.

\begin{cor}
Fix a simplicial graph $\Gamma = (V,E)$ and a compact Hausdorff space $X$.  For each $v \in V$, let $G_v$ be a group, $\alpha_v: G_v \car X$ be an action, and $\tilde{h}_v: G_v\rtimes X\rightarrow \mb{C}$ be a corresponding unital positive definite function.  Suppose that $\left\{\alpha_v\right\}_{v \in V}$ and $\left\{\tilde{h}_v \right\}_{v \in V}$ commute with respect to $\Gamma$.  Then $\bigstar_\Gamma \tilde{h}_v$ is positive definite.
\end{cor}

We now apply Theorem \ref{posdef} to prove the following theorem.

\begin{thm}\label{hp}
For each $v \in V$, let $G_v$ be a discrete group with action $\alpha_v: G_v \car \mc{A}$ such that the actions $\left\{\alpha_v\right\}$ commute according to $\Gamma$. If $\alpha_v$ has the Haagerup property for each $v \in V$ and the witnessing multipliers commute according to $\Gamma$ then $\bigstar_\Gamma \alpha_v$ has the Haagerup property.
\end{thm}

\begin{proof}
It suffices to show this for $|V| < \infty$. For each $v \in V$, let $\left\{h_{v,n}\right\} \in C_0(G_v, \mc{A})$ be a sequence of unital positive definite multipliers such that $\ds ||h_{v,s}|| \leq \frac{1}{2}$ for $s \in G_v \setminus \left\{e \right\}$ (cf. Proposition \ref{unital}).  We claim that $h_n:= \bigstar_\Gamma h_{v,n}$ vanishes at infinity for each $n \in \mb{N}$.  Let $\varepsilon >0$ be given.  For each $v \in V$ let $F_v \subset G_v$ be a finite subset for which $||h_{v,n,s}||< \varepsilon$ for every $s \in G_v \setminus F_v$ (so $e \in F_v$).  Let $K \in \mb{N}$ be such that $2^{-K} \leq \varepsilon$. Put  \[F= \left\{ s= s_1 \cdots s_m \in \bigstar_\Gamma G_v \text{ reduced: }  m \leq K, s_j \in F_{v_j}\right\}\cup\left\{e\right\}.\]  Then $F \subset \bigstar_\Gamma G_v$ is finite with $||h_{n,s}|| < \varepsilon$ for every $s \in \bigstar_\Gamma G_v \setminus F$.  Thus, $h_n \in C_0(\bigstar_\Gamma G_v, \mc{A})$ for every $n$ and $h_n \rightarrow 1_\mc{A}$ pointwise.
\end{proof}

If each $G_v$ has the classical Haagerup property, then we can take $\mc{A} = \mb{C}$ and all actions to be trivial.  All requisite commuting relations are trivially satisfied, and so Theorem \ref{hp} is a generalization of the main result of \cite{antdre}: the Haagerup property for groups is stable under taking graph products. It should be noted that the stability of the Haagerup property for groups under graph products can also be deduced directly from Corollary 2.33 of \cite{casfim} or Theorem 4.4 of \cite{gpucp}.

We again consider the notion of amenable group actions.  In addition to the definition given in Example \ref{ultraexample}, recall that the amenability of a group action $\alpha: G \car \mc{A}$ can also be characterized by the existence of a sequence of finitely supported positive definite multipliers $h_n: G \rightarrow \mc{Z}(\mc{A})$ converging pointwise to $1_\mc{A}$ (cf. \cite{ad,brownozawa,donrua}).  Thus amenable actions have the Haagerup property, and Theorem \ref{hp} immediately yields the following corollary.

\begin{cor}
For each $v \in V$, let $G_v$ be a discrete group with action $\alpha_v: G_v \car \mc{A}$ such that the actions $\left\{\alpha_v\right\}$ commute according to $\Gamma$. If $\alpha_v$ is amenable for each $v \in V$ and the witnessing multipliers commute according to $\Gamma$ then $\bigstar_\Gamma \alpha_v$ has the Haagerup property.
\end{cor}

%\noin In particular, with Example \ref{ultraexample} in mind, we can drop the $\Gamma$-commuting assumption on the multipliers at the cost of considering ultrapowers.
%
%\begin{cor}
%For each $v \in V$, let $G_v$ be a discrete group with action $\alpha_v: G_v \car \mc{A}$ such that the actions $\left\{\alpha_v\right\}$ commute according to $\Gamma$. Let $\mc{U}$ be a free ultrafilter on $\mb{N}$. Let $\alpha_{v,\mc{U}}: G_v \car \mc{A}_\mc{U}$ denote the diagonal actions as discussed in Example \ref{ultraexample}.  If $\alpha_v$ is amenable for each $v \in V$, then $\bigstar_\Gamma \alpha_{v,\mc{U}}$ has the Haagerup property. 
%\end{cor}
%
%\begin{proof}
%This follows from the observation that the coset of the sequence of \tql averaged\tqr multipliers $(h_i)_\mc{U}$ as constructed in Example \ref{ultraexample} is finitely supported (and thus vanishes at $\infty$) if the original multiplier is finitely supported.
%\end{proof}

Continuing the discussion on left transformation groupoids, the corresponding analog for Theorem \ref{hp} is as follows.

\begin{cor}
If for each $v \in V$, the left transformation groupoid $G_v \rtimes X$ has the Haagerup property and the corresponding actions and positive definite functions commute according to $\Gamma$, then $(\bigstar_\Gamma G_v)\rtimes X$ also has the Haagerup property.
\end{cor}

\noin Moreover, due to Theorem 4.2 of \cite{donrua}, we immediately see that Theorem \ref{hp} implies the following corollary.  
\begin{cor}[\cite{dadgue}]\label{ce}
Coarse embeddability of a group into a Hilbert space is stable under graph products.
\end{cor}

\noin We have attributed Corollary \ref{ce} to \cite{dadgue} because graph products can be perceived as amalgamated free products, and Dadarlat-Guentner showed in Theorem 5.1 of \cite{dadgue} that coarse embeddability is preserved under amalgamated free products.

\section{Obstructions to alternative and generalized approaches}\label{ob}

\subsection{Graph products as amalgams} As mentioned above, it is well known that graph products can be expressed as amalgamated free products--see Lemma 3.20 of \cite{green} or the \tql unscrewing technique\tqr in \cite{casfim}; so in situations in which the direct product and amalgamated free product cases are known, one can take care to arrange an argument in conjunction with this decomposition to prove the corresponding graph products version.  For example, the main result of \cite{gpucp} (graph products of unital completely positive maps are completely positive) can be obtained in this manner. Since the Haagerup property is not preserved under taking amalgamated free products in general, we have reason to avoid perceiving graph products as amalgams in this article (save for attributing Corollary \ref{ce} to \cite{dadgue}).

%The author's interest in establishing certain combinatorial ideas for graph products led to the proof strategy appearing in \cite{gpucp}.  

\subsection{Generalization to Hilbert $\mc{A}$-modules}

In \cite{donrua}, the Haagerup property for group actions is considered as an instance of the more general \emph{Hilbert $\mc{A}$-module Haagerup property}.  Thus, it is natural to ask if our treatment of graph products can be generalized to the context of Hilbert $\mc{A}$-modules.  The setting is as follows.

Let $\mc{A}$ be a unital $C^*$-subalgebra of a unital $C^*$-algebra $\mc{B}$ such that there exists a faithful conditional expectation $\mc{E}: \mc{B} \rightarrow \mc{A}$. The conditional expectation $\mc{E}$ gives rise to an $\mc{A}$-valued inner product $\langle x |y\rangle_\mc{E} = \mc{E}(y^*x)$ on $\mc{B}$, and thus we may consider $\mc{B}$ as a Hilbert $\mc{A}$-module.

\begin{dfn}[\cite{donrua}]
In the setting above, $\mc{B}$ has the \emph{Hilbert $\mc{A}$-module Haagerup property} with respect to $\mc{E}$ if there exists a sequence of completely positive $\mc{A}$-bimodule maps $\left\{\Phi_n \right\}$ on $\mc{B}$ such that 
\begin{itemize}
\item $\mc{E} \circ \Phi_n \leq \mc{E}$;
\item each $\Phi_n$ defines a compact $\mc{A}$-module map $\tilde{\Phi}_n$ on $\mc{H}_\mc{A}$--the appropriate completion of $\mc{A}$ under the inner product $\langle \cdot | \cdot \rangle_\mc{E}$;
\item $||\tilde{\Phi}_n(x) - x||_\mc{E} \rightarrow 0$ for all $x \in \mc{B}$.
\end{itemize}
\end{dfn}

As discussed in Example \ref{fex}, the definition of graph products of multipliers (Definition \ref{gppdm}) is inspired by the free case. Using amalgamated free products of completely positive maps, one can expect that the Hilbert $\mc{A}$-module Haagerup property is preserved under taking free products. Unfortunately, this approach breaks down when commuting relations are introduced in the general graph product setting. 

For the sake of illustration, let us return to the reduced crossed product setting.  Given a group $G$ acting on a $C^*$-algebra $\mc{A}$ with action $\alpha: G \car \mc{A}$ and corresponding positive definite multiplier $h: G \rightarrow \mc{Z}(\mc{A})$, one can obtain a completely positive $\mc{A}$-bimodule map  $\Phi_h: G \rtimes_{\alpha,r} \mc{A} \rightarrow G \rtimes_{\alpha,r} \mc{A}$ associated with $h$ by putting $\Phi_h(\lambda_s) = \lambda_sh_s$ for every $s \in G$ and extending $\mc{A}$-linearly.  This connection between $h$ and $\Phi_h$ provides the dictionary between the Haagerup property for group actions and the Hilbert $\mc{A}$-module Haagerup property.  Now fix a simplicial graph $\Gamma= (V,E)$. If we have groups $\left\{G_v\right\}_{v \in V}$ with actions $\left\{\alpha_v: G_v \car \mc{A}\right\}_{v \in V}$ and corresponding unital positive definite multipliers $\left\{h_v: G_v \rightarrow \mc{Z}(\mc{A})\right\}_{v \in V}$ that both commute according to $\Gamma$, then by Theorem \ref{posdef}, we can form the unital positive definite graph product multiplier $\bigstar_\Gamma h_v: \bigstar_\Gamma G_v \rightarrow \mc{Z}(\mc{A})$ which gives rise to the associated unital completely positive map $\Phi_{\bigstar_\Gamma h_v}: \bigstar_\Gamma G_v \rtimes_{\bigstar_\Gamma \alpha_v,r} \mc{A} \rightarrow \bigstar_\Gamma G_v \rtimes_{\bigstar_\Gamma \alpha_v,r} \mc{A}$.  One is tempted to view this map as  a graph product of the component maps (cf. \cite{gpucp}), evidenced by the following.  Given $s_{i} \in G_{v_i}$ for $1 \leq i \leq n$ with $s_1\cdots s_n \in \bigstar_\Gamma G_v$ reduced, we have 
\begin{align*}
\Phi_{\bigstar_\Gamma h_v}(\lambda_{s_1}\cdots \lambda_{s_n}) &= \Phi_{\bigstar_\Gamma h_v}(\lambda_{s_1\cdots s_n}) \\
&=\lambda_{s_1\cdots s_n} (\bigstar_\Gamma h_v)_{s_1\cdots s_n}\\
&=\Phi_{h_{v_1}}(\lambda_{s_1}) \cdots \Phi_{h_{v_n}}(\lambda_{s_n}).
\end{align*}
\noin A counterpoint to this line of reasoning is that this construction heavily depends on the structure of reduced crossed products.  In particular, the definition of $\Phi_h$ utilizes the fact that the elements in the dense subalgebra $C_c(G, \mc{A})$ can be uniquely expressed as $\sum \lambda_s a_s$ so that merely defining $\Phi_h$ on the $\lambda_s$ elements and extending $\mc{A}$-linearly is enough to give a well-defined map.  Furthermore, the requisite commuting relations to even define the graph product of multipliers also depends on this unique decomposition.  The condition that the multipliers $h_v$ commute according to $\Gamma$ only ensures that $\Phi_{h_v}(\lambda_s)\Phi_{h_w}(\lambda_t) = \Phi_{h_w}(\lambda_t)\Phi_{h_v}(\lambda_s)$ whenever $s \in G_v, t \in G_w$ and $(v,w) \in E$.  Evidently, given $s \in G_v, t \in G_w$ with $(v,w) \in E$ and $a,b \in \mc{A}$,
\begin{align*}
\Phi_v(\lambda_s a) \Phi_w(\lambda_t b) &= \lambda_s h_{v,s} a \lambda_t h_{w,t} b\\
&= \lambda_{st} h_{v,s}h_{w,t} \alpha_t^{-1} (a)b\\
&= \lambda_{st} (\bigstar_\Gamma h_v)_{st} \alpha_t^{-1} (a)b,\\
&\text{and}\\
\Phi_w(\lambda_t b) \Phi_v(\lambda_s a) &= \lambda_t h_{w,t} b \lambda_s h_{v,s} a\\
&= \lambda_{ts} h_{w,t} h_{v,s} \alpha_s^{-1} (b) a\\
&= \lambda_{st} (\bigstar_\Gamma h_v)_{st} \alpha_s^{-1} (b)a.
\end{align*}
Thus $\Phi_v(\lambda_s a)$ and $\Phi_w(\lambda_t b)$ do not commute--not even if $\mc{A}$ is commutative.  In fact, $\lambda_s a$ and $\lambda_t b$ do not commute within $(\bigstar_\Gamma G_v) \rtimes_{\bigstar_\Gamma \alpha_v,r} \mc{A}$. So neither the map $\Phi_{\bigstar_\Gamma h_v}$ nor the algebra $(\bigstar_\Gamma G_v) \rtimes_{\bigstar_\Gamma \alpha_v,r} \mc{A}$ should be construed as a(n amalgamated) graph product.

When we strip away the crossed product structure and consider the general Hilbert $\mc{A}$-module setting, we lose access to a reliable unique decomposition on which we would determine the commuting relations and the behavior of the component unital completely positive $\mc{A}$-bimodule maps $\Phi_v: \mc{B}_v \rightarrow \mc{B}_v$ rendering a graph product of such maps unavailable. 

%As mentioned above, the ucp results cannot be directly applied in the present paper because the graph products are taken at the level of group actions rather than at the level of algebras. 

\section{Proof of Theorem \ref{posdef}}\label{appb}

Given a group $G$, a $C^*$-algebra $\mc{A}$, an action $\alpha: G \car \mc{A}$, and an associated completely bounded multiplier $h: G \rightarrow \mc{Z}(\mc{A})$, we can form the kernel $K: G \times G \rightarrow \mc{Z}(\mc{A})$ given by \[K(x,y) = \alpha_y(h_{x^{-1}y}).\]  Clearly, we see that $h$ is a positive definite multiplier if and only if for any $n \in \mb{N}$ and $x_1,\dots, x_n \in G$ the matrix \[\left[K(x_i,x_j)\right]_{i,j}\] is positive.  In this section, we prove several technical results for the kernel associated to the graph product of unital positive definite multipliers with the goal of proving Theorem \ref{posdef}.  These results are similar to and inspired by the results of \S\S 3.1 and \S\S 3.2 in \cite{gpucp} and therefore those of \cite{boca}.

Fix a simplicial graph $\Gamma = (V,E)$ and a unital $C^*$-algebra $\mc{A}$.  For each $v \in V$, let $G_v$ be a group, $\alpha_v: G_v \car \mc{A}$ be an action, and $h_v: G_v \rightarrow \mc{Z}(\mc{A})$ be a corresponding unital positive definite multiplier.  Suppose that $\left\{\alpha_v\right\}_{v \in V}$ and $\left\{h_v \right\}_{v \in V}$ commute with respect to $\Gamma$.  To simplify notation, put
\begin{align*}
\alpha &:= \bigstar_\Gamma \alpha_v\\
&\text{and}\\
h &:= \bigstar_\Gamma h_v.
\end{align*}
Let $K: \bigstar_\Gamma G_v \times \bigstar_\Gamma G_v \rightarrow \mc{Z}(\mc{A})$ denote the kernel given by $K(x,y) = \alpha_y(h_{x^{-1}y})$.  At the end of this section, we prove that $h$ is positive definite.  First, some definitions are in order.

\begin{dfn}[\cite{boca,gpucp}]\label{complete}
A finite subset $X \subset \mc{W}_\text{red}$ is \emph{complete} if $1 \in X$ and whenever $x_1\cdots x_m \in X$ we have $x_{\sigma(2)}\cdots x_{\sigma(m)} \in X$ and  $x_{\sigma(1)} \cdots x_{\sigma(m-1)} \in X$ for every permutation $\sigma \in S_m$ such that $x_1 \cdots x_m = x_{\sigma(1)} \cdots x_{\sigma(m)}$.  In other words $X$ is complete if it contains the unit and is closed under left and right truncations of any equivalent rearrangements.  Let $\mbf{v}_X:= \left\{\mbf{v} \in \mc{W}_\text{red} | \mbf{v} = \mbf{v}_x \text{ for some } x \in X\right\}$.
\end{dfn}

We can place a partial order $\peq$ on $\mc{W}_\text{red} \cup \left\{ 1\right\}$ with respect to truncation as follows.  For every $x \in \mc{W}_\text{red}, 1 \peq x$; and given $x, y \in \mc{W}_\text{red}, y \peq x$ if either $x = y$ or $x$ truncates (as in Definition \ref{complete}) to $y$.  This order also applies to the words in $V$. Let $Y \subset \mc{W}_\text{red}\cup\left\{1\right\}$ be any finite nonempty subset.  Put \[Y^\peq := \left\{x \in \mc{W}_\text{red}\cup\left\{1\right\} | \exists y \in Y: x \peq y \right\}.\]  Clearly, $Y^\peq$ is complete.

\begin{dfn}[\cite{gpucp}]\label{nclength}
Fix $v_0 \in V$.  Let $\mbf{v} = (v_1,\dots,v_n,v_0)$ be reduced.  We let $\nci \mbf{v} \nci_{v_0}$ denote the \emph{(right-hand) non-commutative length of $\mbf{v}$ with respect to $v_0$}, given by \[\nci \mbf{v} \nci_{v_0} := \text{Card}\Big(\left\{i | 1\leq i \leq n, (v_i,v_0) \notin E\right\}\Big).\]  Note that this counts when $v_0$ is repeated. If $\mbf{v}$ cannot be written with $v_0$ at the right-hand end, put $\nci \mbf{v} \nci_{v_0} = -1$.  If $w \in \bigstar_\Gamma \mc{A}_v$ is reduced, let $\nci w \nci_{v_0} = \nci \mbf{v}_w\nci_{v_0}$.  Given a finite set $X$ of reduced words (of vertices or algebra elements), we define the \emph{(right-hand) non-commutative length of $X$ with respect to $v_0$}, denoted $\nci X \nci_{v_0}$ to be given by \[ \nci X \nci_{v_0} := \max_{w \in X} \nci w \nci_{v_0}.\]
\end{dfn}

%\begin{rmk}\label{lengths}
%Observe that in a free product (graph product over a graph with no edges), the length of a reduced word is always one more than the non-commutative length of a reduced word.
%\end{rmk}

\begin{dfn}[\cite{gpucp}]\label{stdform}
Fix $v_0 \in V$. Let $\mbf{x} \in \mc{W}_\text{red}$ be such that $v_0 \in \mbf{x}$. Suppose $\mbf{y},\mbf{c},\mbf{b}  \in \mc{W}_\text{red}$, satisfy the following properties. 
\begin{itemize}
	\item $\mbf{x} = \mbf{y}\mbf{c}(v_0)\mbf{b}$;
	\item $\mbf{b}$ is the word of smallest length so that $\mbf{y}\mbf{c}(v_0) \peq \mbf{x}$ and $\nci \mbf{y}\mbf{c}(v_0)\nci_{v_0} = \nci \left\{\mbf{x}\right\}^\peq \nci_{v_0}$;
	\item $\mbf{y}$ is the word of smallest length so that $\mbf{y}(v_0) \peq \mbf{x}$ and $\nci \mbf{y}(v_0) \nci_{v_0} = \nci \left\{\mbf{x}\right\}^\peq \nci_{v_0}$.
\end{itemize}
Then we say that $\mbf{x} = \mbf{y}\mbf{c}(v_0)\mbf{b}$ is in \emph{standard form with respect to $v_0$}.
We extend this definition to reduced words of algebra elements.
\end{dfn}

\begin{prop}[\cite{gpucp}]
If $\mbf{x} = \mbf{y}\mbf{c}(v_0)\mbf{b}$ is in standard form with respect to $v_0$, then the words $\mbf{y}, \mbf{c},$ and $\mbf{b}$ are unique.
\end{prop}

Next, we establish some intermediate results for the kernel $K$.

\begin{lem}\label{basic}

\begin{enumerate}
\item If $x = x_1\cdots x_m \in \bigstar_\Gamma G_v$ is reduced, then \[h_x = h_{x_1\cdots x_m} = \alpha_{x_m^{-1}\cdots x_2^{-1}}(h_{x_1})h_{x_2\cdots x_m}.\]

\item If $x=x_1\cdots x_m, y=y_1\cdots y_n \in \bigstar_\Gamma G_v$ are reduced words such that $x^{-1}y$ is reduced then \[ K(x_1\cdots x_m, y_1\cdots y_n) = K(x_1\cdots x_m, x_1\cdots x_{m-1})K(x_1\cdots x_{m-1}, y_1\cdots y_n).\]
\end{enumerate}

\end{lem}

\begin{proof}
(1) follows from a direct inductive argument.

To prove (2), we observe that 
\begin{align*}
K(x_1\cdots x_m, y_1\cdots y_m) & = \alpha_{y_1\cdots y_n}(h_{x_m^{-1}\cdots x_1^{-1} y_1\cdots y_n})\\
& = \alpha_{y_1\cdots y_n}(\alpha_{y_n^{-1}\cdots y_1^{-1}x_1\cdots x_{m-1}}(h_{x_m^{-1}})h_{x_{m-1}^{-1}\cdots x_1^{-1}y_1\cdots y_n})\\
& = \alpha_{x_1\cdots x_{m-1}}(h_{x_m^{-1}}) \alpha_{y_1\cdots y_n}(h_{x_{m-1}^{-1}\cdots x_1^{-1}y_1\cdots y_n})\\
& = K(x_1\cdots x_m, x_1\cdots x_{m-1}) K(x_1\cdots x_{m-1},y_1\cdots y_n).\qedhere
\end{align*}
\end{proof}

\begin{prop}
The kernel $K$ is $*$-symmetric.  That is, for every $x, y \in \bigstar_\Gamma G_v,$ \[K(x,y) = K(y,x)^*.\]
\end{prop}

\begin{proof}
We proceed by induction on $|x^{-1}y|$.
\begin{itemize}

\item[$\bullet$] $|x^{-1}y| = 0$: trivial.

\item[$\bullet$] $|x^{-1}y| = m$:  Let $x^{-1}y= z_1\cdots z_m$ be an expression of $x^{-1}y$ as a reduced word. Then we have
\begin{align}
K(x,y) &= \alpha_y(h_{x^{-1}y})\notag\\
&=\alpha_y(h_{z_1\cdots z_m})\notag\\
& = \alpha_y(\alpha_{(z_2\cdots z_m)^{-1}}(h_{z_1}) h_{z_2\cdots z_m}) \notag\\
& = \alpha_y(\alpha_{y^{-1}xz_1}(h_{z_1}))\alpha_y(h_{z_1^{-1}x^{-1}y}) \notag\\
&=\alpha_{xz_1}(h_{z_1})K(xz_1,y) \notag\\
&= \alpha_x(h_{z_1^{-1}}^*)K(y,xz_1)^* \label{indhyp}\\
&= \alpha_x(h_{z_1^{-1}}^*)\alpha_{xz_1}(h_{y^{-1}xz_1}^*)\notag\\
& = \alpha_x(h_{z_1^{-1}}^*)\alpha_{xz_1}(h_{z_m^{-1}\cdots z_2^{-1}}^*)\notag \\
& = \alpha_x(h_{z_1^{-1}}^*)\alpha_{xz_1}(\alpha_{z_2\cdots z_{m-1}}(h_{z_m^{-1}}^*) \cdots \alpha_{z_2}(h_{z_3^{-1}}^*) h_{z_2^{-1}}^*)\notag\\
& = \alpha_x(\alpha_{z_1\cdots z_{m-1}}(h_{z_m^{-1}}^*)\cdots \alpha_{z_1}(h_{z_2^{-1}}^*)h_{z_1^{-1}}^*)\notag\\
&= \alpha_x(h_{z_m^{-1}\cdots z_1^{-1}}^*)\notag\\
&= \alpha_x(h_{y^{-1}x}^*)\notag\\
&= K(y,x)^*\notag
\end{align}
where \eqref{indhyp} follows from the inductive hypothesis.\qedhere
\end{itemize}
\end{proof}

We can improve Lemma \ref{basic} as follows (cf. Lemmas 3.9 and 3.10 of \cite{gpucp}).
	
\begin{lem}\label{X1crossterms}
Fix $v_0 \in V$ and let $x \in \bigstar_\Gamma G_v$ be a reduced word.  Suppose $x = ycab$ is in standard form with respect to $v_0$ with $a \in G_{v_0}$.  Then \[K(ycab, z) = K(ycab,yc)K(yc,z)\] whenever $z$ satisfies either of the following conditions.
\begin{enumerate}
\item $\nci \left\{z\right\}^\peq \nci_{v_0} < \nci \left\{x\right\}^\peq \nci_{v_0}$
\item $\nci \left\{z\right\}^\peq \nci_{v_0} = \nci \left\{x\right\}^\peq \nci_{v_0}$ but $\mbf{v}_y \neq \mbf{v}_{y'}$ where $z = y'c'a'b'$ is in standard form with respect to $v_0$.
\end{enumerate}

\end{lem}

\begin{proof}
Proof of (1): We proceed by induction on $\nci \left\{x\right\}^\peq \nci_{v_0}$. 

\begin{itemize}
	\item $\nci \left\{x\right\}^\peq \nci_{v_0} = 0$: We proceed by further induction on $|z|$.  
		\begin{itemize}
			\item[\bb] $|z| = 0$:  $z= 1$, and the statement is true due to (2) of Lemma \ref{basic}.  
			
			\item [\bb] $|z| = k >0$: if $b^{-1}a^{-1}c^{-1}y{-1}z$ is reduced then the equality holds thanks to (2) of Lemma \ref{basic}.  Suppose $b^{-1}a^{-1}c^{-1}y^{-1}x'$ is not reduced.  In this case $y =e$. Let $c = c_1\cdots c_m$ and $z = z_1\cdots z_k$.  By the definition of standard form, we have that we can rearrange the $c_i$'s and $z_i$'s so that $\mbf{v}_{c_1} = \mbf{v}_{z_1}$.  That is, none of the $b$ terms can cross past $a$; otherwise the minimality of $|b|$ would be contradicted.  If $c_1 \neq z_1$, we have
\begin{align}
&K(ycab,z) \notag\\
&= \alpha_{z_1\cdots z_k}(h_{b^{-1}a^{-1}c_m^{-1}\cdots c_2^{-1}(c_1^{-1}z_1)z_2\cdots z_k})\notag\\
 & = \alpha_{z_1}(\alpha_{z_2\cdots z_k}(h_{b^{-1}a^{-1}c_m^{-1}\cdots c_2^{-1}(c_1^{-1}z_1)z_2\cdots z_k}))\notag\\
 &=  \alpha_{z_1}(K((z_1^{-1}c_1)c_2\cdots c_mab, z_2\cdots z_k))\notag\\
 &=\alpha_{z_1}(K((z_1^{-1}c_1)c_2\cdots c_mab, (z_1^{-1}c_1)c_2\cdots c_m)K((z_1^{-1}c_1)c_2\cdots c_m, z_2\cdots z_k))  \label{1}\\
 &= K(c_1c_2\cdots c_mab, c_1c_2\cdots c_m)K(c_1c_2\cdots c_m, z_1\cdots z_k)\notag 
\end{align}
where \eqref{1} follows from both the fact that \[\nci \left\{z_2 \cdots z_k\right\}^\peq\nci_{v_0} < \nci \left\{(z_1^{-1}c_1) c_2 \cdots c_m a b \right\}^\peq\nci_{v_0}\] and the inductive hypothesis.  The case where $c_1 = z_1$ follows from applying the inductive hypothesis to the fact that  $\nci \left\{z_2 \cdots z_k\right\}^\peq\nci_{v_0} < \nci \left\{c_2 \cdots c_m a b \right\}^\peq\nci_{v_0}$.
		\end{itemize}
		
	\item $\nci \left\{x\right\}^\peq \nci_{v_0} >0$:  Again we induct further on $|z|$.  
		\begin{itemize}
			\item[\bb] $|z| =0$: Again, this follows from (2) of Lemma \ref{basic}.  
			
			\item[\bb] $|\mbf{x}'| = k>0$:  If $b^{-1}a^{-1}c^{-1}y^{-1}z$ is reduced then the equality holds thanks to Lemma \ref{basic}.  Suppose $b^{-1}a^{-1}c^{-1}y^{-1}$ is not reduced, and let $cy = w_1 \cdots w_m$ and $z = z_1 \cdots z_k$.  As before, we can rearrange the $w_i$'s and $z_i$'s so that $\mbf{v}_{w_1} = \mbf{v}_{z_1}$.  If $(\mbf{v}_{w_1}, v_0) \in E$  then the argument in the $\nci \left\{x\right\}^\peq \nci_{v_0} = 0$ case holds.  Assume that $(\mbf{v}_{w_1}, v_0) \notin E$.  Then $\nci w_2 \cdots w_m a\nci_{v_0} = \nci yca\nci_{v_0} - 1 \geq 0$.  It is a quick check to see that if $\nci \left\{z\right\}^\peq\nci_{v_0} \neq -1$ then deleting $z_1$ from the left decreases the non-commutative length by one, and if $\nci \left\{z\right\}^\peq\nci_{v_0} = -1$, then deleting $z_1$ leaves the non-commutative length alone.  In either case, the inductive hypothesis applies, yielding the equality as illustrated above.
		\end{itemize}
	\end{itemize}
	
	Proof of (2): Let $yc = w_1\cdots w_m$ and $y'c' = w'_1\cdots w'_{m'}$.  We proceed by induction on $\nci \left\{x\right\}^\peq\nci_{v_0}$.
\begin{itemize}
	\item $\nci \left\{x\right\}^\peq \nci_{v_0} = 1$: We induct further on $m + m'$.
	
	\begin{itemize}
		\item[$\bullet \bullet$] $m+m' = 2$: Since $\mbf{v}_y \neq \mbf{v}_{y'}$, we immediately get that $b^{-1}a^{-1}w_1^{-1}w_1'a'b'$ is reduced.  So the equality follows from Lemma \ref{basic}.
		
		\item[$\bullet \bullet$] $m + m' > 2$: If $b^{-1}a^{-1}c^{-1}y^{-1}y'c'a'b'$ is reduced then we are done.  Suppose $b^{-1}a^{-1}c^{-1}y^{-1}y'c'a'b'$ is not reduced.  Then we can rearrange the $w$ and $w'$ terms so that $\mbf{v}_{w_1} = \mbf{v}_{w'_1}$.  
		Then we have
		\begin{align}
		&K(w_1\cdots w_m a b, w'_1 \cdots w'_{m'} a' b') \notag\\
		& = \alpha_{w'_1\cdots w'_{m'}ab}(h_{b^{-1}a^{-1}w_m^{-1}\cdots w_2^{-1}(w_1^{-1}w'_1)w'_2\cdots w'_{m'}ab}) \notag\\
		& = \alpha_{w'_1}(\alpha_{w'_2\cdots w'_{m'}ab}(h_{b^{-1}a^{-1}w_m^{-1}\cdots w_2^{-1}(w_1^{-1}w'_1)w'_2\cdots w'_{m'}ab}))\notag\\
		& = \alpha_{w'_1}(K(({w'_1}^{-1}w_1)w_2\cdots w_m ab, w'_2\cdots w'_{m'} ab)	\label{2}
		\end{align}
		Since $\mbf{y} \neq \mbf{y}'$ we have that $(\mbf{v}_{w_1}, v_0) \in E$.   The inductive hypothesis on $m + m'$ applies, yielding the desired equality.
	\end{itemize}
	
	\item $\nci \left\{x\right\}^\peq \nci_{v_0} >1$: Again, induct further on $m + m'$.
	
		\begin{itemize}
		
			\item[$\bullet \bullet$] $m + m' = 2\nci \left\{x\right\}^\peq \nci_{v_0}$: Suppose $b^{-1}a^{-1}c^{-1}y^{-1}y'c'a'b'$ is not reduced and that $\mbf{v}_{w_1} = \mbf{v}_{w'_1}$.  Then we obtain the same decomposition as in \eqref{2}.  Then by applying part (1) to the case where $w_1 \neq w_1'$ and the inductive hypothesis to the $w_1 = w_1'$ case, we obtain the desired equality.

			\item[$\bullet \bullet$] $m + m' > 2\nci\left\{x\right\}^\peq\nci_{v_0}$: Suppose $b^{-1}a^{-1}c^{-1}y^{-1}y'c'a'b'$ is not reduced and that $\mbf{v}_{w_1} = \mbf{v}_{w'_1}$; consider the decomposition from \eqref{2}.  If $(\mbf{v}_{w_1}, v_0) \notin E$, then as in the $m + m' = 2\nci \left\{x\right\}^\peq \nci_{v_0}$ case, apply part (1) to the $w_1 \neq w_1'$ case and the inductive hypothesis to the $w_1 = w_1'$ case.  If $(\mbf{v}_{w_1}, v_0) \in E$, apply the inductive hypothesis on $m + m'$.\qedhere		
		\end{itemize}	
\end{itemize}

\end{proof}

As in \cite{boca,gpucp}, we consider $\mc{A}$ to be a unital $C^*$-subalgebra of $B(\mc{H})$ for a Hilbert space $\mc{H}$. It will suffice to show that for any finite subset $X \subset \bigstar_\Gamma G_v$ and any function $\xi: X \rightarrow \mc{H}$, we have 
\begin{align}
\sum_{x,y \in X} \langle K(x,y)\xi(y) | \xi(x)\rangle& \geq 0. \label{goal}
\end{align}
Since every finite set is contained in a finite complete set, it suffices to show \eqref{goal} for every complete set $X$ and function $\xi: X\rightarrow \mc{H}$.

We will make use of a Stinespring construction for the present context. Let $X\subset \bigstar_\Gamma G_v$ be a complete set and consider $\mb{C}^{|X|}$ with standard basis $\left\{e_x\right\}_{x\in X}$.  The inequality \eqref{goal} implies that we can define a positive semi-definite sesquilinear form on $\mc{H} \otimes \mb{C}^{|X|}$ given by \[ \langle \xi \otimes e_y | \eta \otimes e_x\rangle = \langle K(x,y) \xi | \eta\rangle.\]  By standard arguments this yields a Hilbert space  that we will denote by $\mc{H} \otimes_K \mb{C}^{|X|}$.  For each $x \in X$ let $V_x: \mc{H} \rightarrow \mc{H}\otimes_K \mb{C}^{|X|}$ be given by $V_x(\xi) = \xi \otimes_\Theta e_x$.  Because $h$ is unital, one sees immediately that $V_x$ is an isometry for any $x \in X$.

Given $x \in X$ with $|x| = 1$, we define the \emph{left-concatenation operator} $L_x: \mc{H}\otimes_\Theta \mb{C}^{|X|} \rightarrow \mc{H} \otimes_\Theta \mb{C}^{|X|}$ as follows. \[L_x(\xi \otimes_\Theta e_y) = \left\{\begin{array}{lcr}
0 & \text{if} & xy \notin X \\
&&\\
\xi \otimes_\Theta e_{xy} & \text{if} & xy \in X
\end{array}\right.\]  Observe that $L_x$ is bounded:
\begin{align*}
||L_x (\xi\otimes e_y)||^2_{\mc{H}\otimes_K \mb{C}^{|X|}} &\leq \langle \xi \otimes_K e_{xy} | \xi \otimes_K e_{xy}\rangle \\
& = \langle K(xy,xy) \xi | \xi\rangle\\
&= ||\xi||^2_\mc{H}.
\end{align*}

As in \cite{gpucp}, we have a version of the Schwarz inequality for our situation.

\begin{prop}\label{schwarz}
Let $X \subset \mc{W}_\text{red}$ be a complete set, and assume that for every function $\xi: X \rightarrow \mc{H}$, \eqref{goal} holds. For $1 \leq i \leq N$, let $c_i, b_i, c_ib_i \in X$. If additionally we have $K(c_ib_i,c_j) = K(c_ib_i,c_i)K(c_i,c_j)$ for every $1 \leq i, j \leq N$, then we have the following matrix inequality. \[ \left[ K(c_ib_i,c_jb_j)\right]_{ij} \geq \left[K(c_ib_i,c_i)K(c_i,c_j)K(c_j,c_jb_j)\right]_{ij}\]
\end{prop}

\begin{proof}
Given $x,y \in X,$ \[K(x,y) = V_e^*L_x^*L_yV_e.\]  

%Observe that \[V_eV_e^*(\xi \otimes_K e_x) = (K(e,x)\xi)\otimes_\Theta e_e.\]  

Our goal is to show \[\left[V_e^*L_{b_i}^*L_{c_i}^*L_{c_j}L_{b_j}V_e\right]_{ij}  \geq \left[V_e^*L_{b_i}^*L_{c_i}^*L_{c_i}V_eV_e^*L_{c_i}^*L_{c_j}V_eV_e^*L_{c_j}^*L_{c_j}L_{b_j}V_e\right]_{ij},\] or equivalently \[\sum_{i,j=1}^N \langle (V_e^*L_{b_i}^*L_{c_i}^*L_{c_j}L_{b_j}V_e - V_e^*L_{b_i}^*L_{c_i}^*L_{c_i}V_eV_e^*L_{c_i}^*L_{c_j}V_eV_e^*L_{c_j}^*L_{c_j}L_{b_j}V_e) \xi_j | \xi_i \rangle \geq 0\] for any $\xi_1,\dots,\xi_N \in \mc{H}$.  First, for any $1\leq i,j\leq N$, consider the following equality.

\begin{align*}
&\langle L_{c_i}^*L_{c_j} V_eV_e^*L_{c_j}^*L_{c_j} (\xi_j \otimes_K e_{b_j}) | (I - V_eV_e^*L_{c_i}^*L_{c_i})(\xi_i \otimes_K e_{b_i})\rangle\\
 & = \langle (K(c_j, c_jb_j)\xi_j) \otimes_K e_{c_j}|\xi_i \otimes_K e_{c_ib_i}\rangle \\
 &- \langle (K(c_j,c_jb_j) \xi_j) \otimes_K e_{c_j} | (K(c_i,c_ib_i) \xi_i) \otimes_K e_{c_i}\rangle\\
& = \langle K(c_ib_i,c_j)K(c_j,c_jb_j)\xi_j|\xi_i\rangle - \langle K(c_i,c_j)K(c_j, c_jb_j)\xi_j|K(c_i, c_ib_i)\xi_i\rangle\\
& = \langle K(c_ib_i,c_j)K(c_j,c_jb_j)\xi_j|\xi_i\rangle - \langle K(c_ib_i, c_i)K(c_i,c_j)K(c_j, c_jb_j)\xi_j|\xi_i\rangle\\
&= 0
\end{align*}
\noin Thus we have 

\begin{align*}
&\sum_{i,j=1}^N \langle (V_e^*L_{b_i}^*L_{c_i}^*L_{c_j}L_{b_j}V_e - V_e^*L_{b_i}^*L_{c_i}^*L_{c_i}V_eV_e^*L_{c_i}^*L_{c_j}V_eV_e^*L_{c_j}^*L_{c_j}L_{b_j}V_e) \xi_j | \xi_i\rangle\\
 & = \sum_{i,j=1}^N \Big(\langle L_{c_i}^*L_{c_j} (\xi_j \otimes_K e_{b_j}) | \xi_i \otimes_K e_{b_i}\rangle \\
 &- \langle L_{c_i}^*L_{c_j} V_eV_e^*L_{c_j}^*L_{c_j} (\xi_j \otimes_K e_{b_j}) | V_eV_e^*L_{c_i}^*L_{c_i}(\xi_i\otimes_K e_{b_i})\rangle\Big)\\
& = \sum_{i,j =1}^N\langle L_{c_i}^*L_{c_j} (I-V_eV_e^*L_{c_j}^*L_{c_j}) (\xi_j \otimes_K e_{b_j}) | (I-V_eV_e^*L_{c_i}^*L_{c_i})(\xi_i \otimes_K e_{b_i})\rangle\\
& + 2\frak{Re}\langle L_{c_i}^*L_{c_j} V_eV_e^*L_{c_j}^*L_{c_j} (\xi_j \otimes_K e_{b_j}) |(I- V_eV_e^*L_{c_i}^*L_{c_i})(\xi_i\otimes_K e_{b_i})\rangle\\
&= \sum_{i,j =1}^N\langle L_{c_j} (I-V_eV_e^*L_{c_j}^*L_{c_j}) (\xi_j \otimes_K e_{b_j}) | L_{c_i}(I-V_eV_e^*L_{c_i}^*L_{c_i})(\xi_i \otimes_K e_{b_i})\rangle\\
& \geq 0.\qedhere
\end{align*}
\end{proof}

\noin This implies the positive definite multiplier analog of Schwarz's inequality: for any sequence $a_1,\dots a_n \in G_{v_0},$ \[\left[\alpha_{a_j}(h_{a_i^{-1}a_j})\right]_{ij} = \left[K(a_i,a_j)\right]_{ij} \geq \left[K(a_i,e)K(e,a_j)\right]_{ij} = \left[h_{a_i^{-1}}\alpha_{a_j}(h_{a_j})\right]_{ij}.\]

\begin{lem}\label{Y1square}
Let $\left\{ x_i\right\}_{i=1}^N \in (\mc{W}_\text{red})^N$ be a finite sequence such that for every $1 \leq i \leq N,$ we have  $v_0 \in \mbf{v}_{x_i}$.  For each $1 \leq i \leq N$, let $x_i = y_ic_ia_ib_i$ be in standard form with respect to $v_0$ ($a_i \in G_{v_0}$).  Assume the following.
\begin{enumerate}
	\item For every $1 \leq i,j \leq N, \mbf{v}_{y_i} = \mbf{v}_{y_j}$;
	\item\label{sc} For every complete set $X \subsetneq (\left\{ x_i\right\}_{i=1}^N)^\peq$ and any function $\xi: X \rightarrow \mc{H}$, \eqref{goal} holds.
\end{enumerate}
Then \[\left[K(x_i,x_j)\right]_{ij} \geq \left[K(y_ic_ia_ib_i,y_ic_i)K(y_ic_i,y_jc_j)K(y_jc_j,y_jc_ja_jb_j)\right]_{ij}.\]
\end{lem}

\begin{proof}\hspace*{\fill}
%We proceed by induction on $|(\left\{x_i\right\}_{i=1}^N)^\peq|$.

%\begin{itemize}
%	\item $|(\left\{x_i\right\}_{i=1}^N)^\peq| =2$: We have $(\left\{x_i\right\}_{i=1}^N)^\peq = \left\{1, a\right\}$.  So $x_i = a \in \mathring{\mc{A}}_{v_0}$ for every $1 \leq i \leq N$, and by the classical version of Schwarz's Inequality \[\left[\Theta(x_i^*x_j)\right]_{ij} = \left[\theta_{v_0}(x_i^*x_j)\right]_{ij} \geq \left[\theta_{v_0}(x_i^*)\theta_{v_0}(x_j)\right]_{ij} = \left[\Theta(x_i^*)\Theta(x_j)\right]_{ij}.\]
%	
%	\item $|(\left\{x_i\right\}_{i=1}^N)^\peq| >2$:
	\begin{itemize}
		\item First suppose $\nci (\left\{x_i\right\}_{i=1}^N)^\peq\nci_{v_0} = 0$. Then for every $1 \leq i \leq N, y_i = e$.  So $x_i = c_ia_ib_i$.  Standard form implies that for each $1 \leq i \leq N, c_i$ commutes with $a_i$.  Note that 
		\begin{align}
		K(c_ia_ib_i,c_ja_jb_j) & = \alpha_{c_ja_jb_j}(h_{b_i^{-1}a_i^{-1}c_i^{-1}c_ja_jb_j})\notag\\
		& = \alpha_{c_ja_jb_j}(h_{b_i^{-1}(a_i^{-1}a_j)c_i^{-1}c_jb_j})\notag\\
		& = \alpha_{c_ja_jb_j}(\alpha_{b_j^{-1}c_j^{-1}c_i}(h_{b_i^{-1}(a_i^{-1}a_j)})h_{c_i^{-1}c_jb_j})\label{4.1}\\
		& = \alpha_{c_ja_jb_j}(\alpha_{b_j^{-1}c_j^{-1}c_i}(\alpha_{a_j^{-1}a_i}(h_{b_i^{-1}})h_{a_i^{-1}a_j})\alpha_{b_j^{-1}}(h_{c_i^{-1}c_j})h_{b_j})\label{4.2}\\
		& = \alpha_{c_ia_i}(h_{b_i^{-1}})\alpha_{c_ia_j}(h_{a_i^{-1}a_j})\alpha_{c_ja_j}(h_{c_i^{-1}c_j})\alpha_{c_ja_jb_j}(h_{b_j})\notag\\
		& = K(c_ia_ib_i,c_ia_i)K(c_ia_i,c_ia_j)K(c_ia_j,c_ja_j)K(c_ja_j,c_ja_jb_j)\notag
		\end{align}
		where \eqref{4.1} and \eqref{4.2} follow from Lemma \ref{X1crossterms}.  Also by Lemma \ref{X1crossterms}, we have
		\begin{align}
		&K(c_ia_ib_i,c_i)K(c_i,c_j)K(c_j,c_ja_jb_j) \notag\\
		&= K(c_ia_ib_i,c_ia_i)K(c_ia_i, c_i)K(c_i,c_j)K(c_j,c_ja_j)K(c_ja_j,c_ja_jb_j).\notag
		\end{align}  Thus it suffices to show
		\[\left[K(c_ia_i,c_ia_j)K(c_ia_j,c_ja_j)\right]_{ij} \geq \left[K(c_ia_i, c_i)K(c_i,c_j)K(c_j,c_ja_j)\right]_{ij}.\]
		
		If we take $h_{v,s} \in \mc{Z}(\mc{A})^+$ for $v \in V$ and $s \in G_v$, then the first condition for commuting multipliers implies that $\alpha_{a}(h_b) = h_b$ for any $a \in G_v$ and $ b \in G_w$ with $(v,w) \in E$.  Then we have
		
		\begin{align}
		\left[K(c_ia_i,c_ia_j)K(c_ia_j,c_ja_j)\right]_{ij} & = \left[\alpha_{c_ia_j}(h_{a_i^{-1}a_j})\alpha_{c_ja_j}(h_{c_i^{-1}c_j})\right]_{ij}\notag\\
		& = \left[\alpha_{a_j}(h_{a_i^{-1}a_j})\alpha_{c_j}(h_{c_i^{-1}c_j})\right]_{ij}\notag\\
		&\geq \left[h_{a_i^{-1}}\alpha_{c_j}(h_{c_i^{-1}c_j})\alpha_{a_j}(h_{a_j})\right]_{ij}\label{tak}\\
		& = \left[\alpha_{c_i^{-1}}(h_{a_i^{-1}})\alpha_{c_j}(h_{c_i^{-1}c_j})\alpha_{c_ja_j}(h_{a_j})\right]_{ij}\notag\\
		& = \left[K(c_ia_i, c_i) K(c_i,c_j)K(c_j,c_ja_j)\right]_{ij}\notag
		\end{align}
		where \eqref{tak} follows from Schwarz's inequality for positive definite multipliers and Lemma IV.4.24 in \cite{takesaki}.

		\item Now suppose that $\nci (\left\{x_i\right\}_{i=1}^N)^\peq \nci_{v_0} >0$.  Say that $y_i = y_1(i) \cdots y_m(i)$.  If $y_1(i) \neq y_1(j)$, observe that 
		\begin{align}
		&K(y_ic_ia_ib_i,y_jc_ja_jb_j) \notag\\
		&=\alpha_{y_jc_ja_jb_j}(h_{b_i^{-1}a_i^{-1}c_i^{-1}y_i^{-1}y_jc_ja_jb_j})\notag \\
		& =\alpha_{y_jc_ja_jb_j}(h_{b_i^{-1}a_i^{-1}c_i^{-1}y_m(i)^{-1}\cdots y_2(i)^{-1}(y_1(i)^{-1}y_1(j))y_2(j)\cdots y_m(j)c_ja_j b_j})\notag\\
		&= \alpha_{y_jc_ja_jb_j}(\alpha_{b_{j}^{-1}a_j^{-1}c_j^{-1}y_j^{-1}y_ic_i}(h_{b_i^{-1}a_i^{-1}})h_{c_i^{-1}y_i^{-1}y_jc_ja_jb_j})\label{6} \\
		& = \alpha_{y_ic_i}(h_{b_i^{-1}a_i^{-1}})\alpha_{y_jc_ja_jb_j}(\alpha_{b_j^{-1}a_j^{-1}}(h_{c_i^{-1}y_i^{-1}y_jc_j})h_{a_jb_j})\label{6.1}\\
		& = \alpha_{y_ic_i}(h_{b_i^{-1}a_i^{-1}})\alpha_{y_jc_j}(h_{c_i^{-1}y_i^{-1}y_jc_j})\alpha_{y_jc_ja_jb_j}(h_{a_jb_j})\notag\\
		& = K(y_ic_ia_ib_i,y_ic_i)K(y_ic_i,y_jc_j)K(y_jc_j,y_jc_ja_jb_j)\notag
		\end{align}
		where \eqref{6} and \eqref{6.1} follow from Lemma \ref{X1crossterms}.  In case $y_1(i) = y_1(j)$, we note that by Lemma \ref{X1crossterms}, \begin{align*}& K(y_2(i)\cdots y_m(i)c_ia_ib_i,y_2(j)\cdots y_m(j) c_j)\\& = K(y_2(i)\cdots y_m(i)c_ia_ib_i,y_2(i)\cdots y_m(i)c_i) K(y_2(i)\cdots y_m(i)c_i,y_2(j)\cdots y_m(j) c_j).\end{align*} So, since $(\left\{ y_2(i)\cdots y_m(i)c_ia_ib_i\right\}_{i=1}^N)^\peq$ is a strictly smaller complete set, then assumption \eqref{sc} combined with Proposition \ref{schwarz} in conjunction with the $y_1(i)\neq y_1(j)$ case gives the desired inequality.\qedhere
	\end{itemize}
%\end{itemize}
\end{proof}

\begin{proof}[Proof of Theorem \ref{posdef}]  We wish to show that for every complete set $X$ and function $\xi: X \rightarrow \mc{H}$, \eqref{goal} holds. We proceed by induction on $|X|$.

\begin{itemize}
	\item $|X| =1$: Trivial.
	
	\item $|X| \geq 2$: Let $(v_0) \in \mbf{v}_X$. Put \[X_1:=\left\{ x \in X \Big| \nci \left\{x\right\}^\peq \nci_{v_0} = \nci X \nci_{v_0}\right\},\] and let $x_0 \in X_1$ be an element of longest length in $X_1$.  Say that $x_0 = y_0c_0a_0b_0$ is in standard form with respect to $v_0$ (and so $a_0 \in G_{v_0}$).  Define \[Y_1 := \left\{x \in X_1 \big| \text{ in standard form } x = ycab \,(a \in G_{v_0}), \mbf{v}_y = \mbf{v}_{y_0}\right\}.\]  Note the following decomposition.
	\begin{align*}
	&\sum_{x,y \in X} \langle K(x,y)\xi(y) | \xi(x)\rangle \\
	& =\sum_{w,z \in X\setminus Y_1} \langle K(w,z)\xi(z) | \xi(w)\rangle\\
	& + \sum_{x, x' \in Y_1} \langle K(x,x')\xi(x') | \xi(x)\rangle\\
	& + \sum_{x \in Y_1, z \in X \setminus Y_1} 2\mathfrak{Re} \langle K(x,z)\xi(z) | \xi(x)\rangle.
	\end{align*}
	Consider $X \setminus Y_1 \subset (X\setminus Y_1)^\peq$.  By our choice of $x_0$, we have that $x_0 \notin (X\setminus Y_1)^\peq$, so the inductive hypothesis on $|X|$ applies to the strictly smaller complete set $(X\setminus Y_1)^\peq$.  By the discussion above, there is a Hilbert space $\mc{K}$ and operators $V_w \in B(\mc{H},\mc{K})$ for every $w \in X\setminus Y_1$ such that $V_w^*V_z = K(w,z)$ for every $w,z \in X \setminus Y_1$. 
	
	For $x, x' \in Y_1$, let $x = ycab$ and $x' = y'c'a'b'$ be their standard forms with respect to $v_0$.  By Lemma \ref{X1crossterms}, we have that
	\begin{align*}
	&\sum_{x \in Y_1, z \in X \setminus Y_1} 2\mathfrak{Re} \langle K(x,z)\xi(z) | \xi(x)\rangle\\
	&= \sum_{ycab \in Y_1,  z \in X \setminus Y_1} 2\mathfrak{Re} \langle K(ycab,yc)K(yc,z)\xi(z) | \xi(ycab)\rangle\\
	&= \sum_{ycab \in Y_1, z \in X \setminus Y_1} 2\mathfrak{Re} \langle V_z\xi(z) | V_{yc}K(yc,ycab)\xi(ycab)\rangle.
	\end{align*}
	By Lemma \ref{Y1square}, we have that
	\begin{align*}
	&\sum_{x, x' \in Y_1} \langle K(x,x')\xi(x') | \xi(x)\rangle\\
	& \geq \sum_{x=ycab, x'=y'c'a'b' \in Y_1} \langle K(ycab,yc)K(yc,y'c')K(y'c',y'c'a'b')\xi(y'c'a'b') | \xi(ycab)\rangle\\
	&= \sum_{ycab, y'c'a'b' \in Y_1} \langle V_{y'c'}K(y'c',y'c'a'b')\xi(y'c'a'b') | V_{yc}K(yc,ycab)\xi(ycab)\rangle\\
	&= \Big|\Big|\sum_{ycab \in Y_1} V_{yc}K(yc,ycab)\xi(ycab)\Big|\Big|^2.
	\end{align*} 
	We also have 
	\begin{align*}
	\sum_{w,z \in X\setminus Y_1} \langle K(w,z)\xi(z) | \xi(w)\rangle &= \sum_{w,z \in X\setminus Y_1} \langle V_w^*V_z \xi(z) | \xi(w)\rangle\\
	& = \sum_{w,z \in X\setminus Y_1} \langle V_z \xi(z) | V_w\xi(w)\rangle\\
	&= \Big|\Big|\sum_{w \in X\setminus Y_1} V_w \xi(w)\Big|\Big|^2
	\end{align*}
	Thus we have
	\begin{align*}
	&\sum_{x,y \in X} \langle K(x,y)\xi(y) | \xi(x)\rangle \\
	& =\sum_{w,z \in X\setminus Y_1} \langle K(w,z)\xi(z) | \xi(w)\rangle + \sum_{x, x' \in Y_1} \langle K(x,x')\xi(x') | \xi(x)\rangle \\
	&+ \sum_{x \in Y_1, z \in X \setminus Y_1} 2\mathfrak{Re} \langle K(x,z)\xi(z) | \xi(x)\rangle\\
	&\geq \Big|\Big|\sum_{w \in X\setminus Y_1} V_w \xi(w)\Big|\Big|^2 + \Big|\Big|\sum_{x=ycab \in Y_1} V_{yc}K(yc,ycab)\xi(ycab)\Big|\Big|^2 \\
	&+  \sum_{x=ycab \in Y_1, z \in X \setminus Y_1} 2\mathfrak{Re} \langle V_z\xi(z) | V_{yc}K(yc,ycab)\xi(ycab)\rangle\\
	&= \Big|\Big|\sum_{w \in X \setminus Y_1} V_w \xi(w) + \sum_{x=ycab \in Y_1} V_{yc}K(yc,ycab) \xi(ycab)\Big|\Big|^2\\
	&\geq 0. \qedhere
	\end{align*}
\end{itemize}
\end{proof}

\subsection*{Acknowledgment}  The author is grateful to Ben Hayes for a valuable conversation about these results.

\bibliographystyle{plain}
\bibliography{gpdsbib}{}

\end{document}